\documentclass[10pt]{amsart}

\title [Weight monoids of smooth affine spherical varieties]{Combinatorial characterization of the weight monoids of smooth affine spherical varieties}
  
\author{Guido Pezzini}
\address{Department Mathematik, Emmy--Noether--Zentrum, FAU Erlangen-N\"urnberg, Cauerstr.\ 11, D-91058 Erlangen}
\email{pezzini@math.fau.de}

\author{Bart Van Steirteghem}
\address{Department of Mathematics, Medgar Evers College - City University of New York, 1650 Bedford Ave., Brooklyn, NY 11225, USA}
\email{bartvs@mec.cuny.edu}

\usepackage{amssymb, amsmath}
\usepackage{mathrsfs}
\usepackage{enumerate}
\usepackage{longtable}

\newtheorem{theorem}{Theorem}[section]
\newtheorem{lemma}[theorem]{Lemma}
\newtheorem{proposition}[theorem]{Proposition}
\newtheorem{corollary}[theorem]{Corollary}

\theoremstyle{definition}
\newtheorem{definition}[theorem]{Definition}
\newtheorem{remark}[theorem]{Remark}
\newtheorem{example}[theorem]{Example}
\newtheorem{List}[theorem]{List}

\numberwithin{equation}{section}

\newcommand{\supp}{\mathrm{supp}}
\newcommand{\C}{\mathbb C}
\newcommand{\CC}{\mathbb C}
\newcommand{\Z}{\mathbb Z}
\newcommand{\ZZ}{\mathbb Z}
\newcommand{\N}{\mathbb N}
\newcommand{\Q}{\mathbb Q}
\newcommand{\QQ}{\mathbb Q}
\newcommand{\GG}{\mathbb G}
\newcommand{\GGm}{\mathbb G_\text{m}}
\newcommand{\A}{\mathbf A}

\newcommand{\PP}{\mathbb P}
\newcommand{\s}{\mathscr S}
\newcommand{\GL}{\mathrm{GL}}
\newcommand{\SL}{\mathrm{SL}}
\newcommand{\Spin}{\mathrm{Spin}}
\newcommand{\SO}{\mathrm{SO}}
\newcommand{\Sp}{\mathrm{Sp}}
\newcommand{\Chi}{\mathcal X}

\newcommand{\sA}{\mathsf{A}}
\newcommand{\sB}{\mathsf{B}}
\newcommand{\sC}{\mathsf{C}}
\newcommand{\sD}{\mathsf{D}}
\newcommand{\sE}{\mathsf{E}}
\newcommand{\sF}{\mathsf{F}}
\newcommand{\sG}{\mathsf{G}}

\DeclareMathOperator{\Aut}{Aut}

\DeclareMathOperator{\Div}{div}
\DeclareMathOperator{\Hom}{Hom}

\DeclareMathOperator{\Spec}{Spec}

\DeclareMathOperator{\soc}{soc}

\newcommand{\<}{\langle}
\renewcommand{\>}{\rangle}
\newcommand{\quot}{/\!\!/}

\newcommand{\wm}{\Gamma}
\newcommand{\ms}{\mathrm{M}_{\wm}}
\newcommand{\dw}{\Lambda^+}
\newcommand{\wl}{\Lambda}
\newcommand{\rl}{\Lambda_R}
\newcommand{\sr}{S}

\newcommand{\lat}{\Lambda}  
\newcommand{\col}{\Delta}

\newcommand{\V}{\mathcal{V}}
\newcommand{\D}{\mathcal{D}}

\newcommand{\osig}{\overline{\sigma}}

\newcommand{\inn}{\subset}

\newcommand{\loccit}{{\em loc.cit.}}

\begin{document}

\begin{abstract}
Let $G$ be a connected complex reductive group. A well known theorem of I.~Losev's says that a smooth affine spherical $G$-variety $X$ is uniquely determined by its weight monoid, which is the set of irreducible representations of $G$ that occur in the coordinate ring of $X$. In this paper, we use the combinatorial theory of spherical varieties and a smoothness criterion of R.~Camus to characterize the weight monoids of smooth affine spherical varieties. 
\end{abstract}

\maketitle

\section{Introduction and main results} \label{sec:intro}
A natural invariant of a complex affine algebraic variety $X$ equipped with an action of a connected reductive group $G$ is its \textbf{weight monoid} $\wm(X)$. By definition, it is the set of (isomorphism classes of) irreducible representations of $G$ that occur in the coordinate ring $\C[X]$ of $X$. In the 1990s, F. Knop conjectured that if $X$ is a smooth affine spherical variety ---i.e.\ if $X$ is smooth and $\C[X]$ is multiplicity free as a representation of $G$--- then $\wm(X)$ determines $X$ up to equivariant automorphism. This conjecture was proved by I.~Losev in \cite{losev-knopconj}. By work of Knop's~\cite{knop-autoHam} it implies that multiplicity free (real) Hamiltonian manifolds (cf.~\cite{guill&stern-mf}) are classified by their moment polytope and generic isotropy group. 

In this paper,  we use the combinatorial theory of spherical varieties and a smoothness criterion of R.~Camus~\cite{camus} to characterize the weight monoids of smooth affine spherical varieties. Our most general statement is Theorem~\ref{thm:general}. In this introduction we give a special case which is more elementary: in Theorem~\ref{thm:main} we characterize the $G$-saturated weight monoids of smooth affine spherical varieties (see Definition~\ref{def:G-saturated}). 

As an application, we characterize in Theorem~\ref{thm:model} when a semisimple and simply connected group $G$ has a smooth affine model variety, i.e.\ a smooth affine $G$-variety in whose coordinate ring all irreducible representations of $G$ occur with multiplicity one.

We point out that, for any given candidate weight monoid $\wm$, our criterion only requires finitely many elementary verifications. In fact, Theorem~\ref{thm:general} can be implemented as an algorithm that given a set of generators of $\wm$ decides whether $\wm$ is the weight monoid of a smooth affine spherical variety. As part of his forthcoming PhD thesis, Won Geun Kim has already implemented the case where $G = \SL(n)$ and $\wm$ is $G$-saturated (as in Theorem~\ref{thm:main}) and free.

Furthermore, thanks to \cite[Theorem 11.2]{knop-autoHam}, our main result gives a local combinatorial characterization of the moment polytopes of (real) multiplicity free Hamiltonian manifolds: by repeating the verifications of our criterion at every vertex of a candidate moment polytope $\mathcal{P}$, one can decide whether $\mathcal{P}$ is the momentum image a multiplicity free Hamiltonian manifold.

We now describe our main result in the special case of $G$-saturated monoids. Fix a
Borel subgroup $B$ of $G$ and a maximal torus $T$ contained in $B$, denote by $S$ the corresponding set of simple roots. Let $U$
be the unipotent radical of $B$. When $\alpha$ is a root of $(G,T)$, we will use $\alpha^{\vee}$ for the corresponding coroot. The
weight lattice of $G$ is denoted $\wl$. Recall that $\wl$ is the
character group of $T$, which we identify with the character
group of $B$. The set of dominant weights of $G$ with respect to $B$
will be denoted $\dw$. Then $\dw$ is a finitely generated
submonoid of $\wl$. We denote by $V(\lambda)$ the irreducible $G$-module
of highest weight $\lambda \in \dw$. Given an affine $G$-variety $X$ we identify its weight monoid $\wm(X)$
with a submonoid of $\dw$:
\[
\wm(X)= \{\lambda \in \dw \colon \Hom_G(V(\lambda), \C[X]) \neq 0\}.
\]
For a subset $\mathcal{E}$ of $\wl$, we will write $\Z\mathcal{E}$ for the sublattice of $\wl$ spanned by $\mathcal{E}$. We will use $\N \mathcal{E}$ for the submonoid (including $0$) of $\wl$ generated by $\mathcal{E}$. If $\Chi$ is a lattice, then we will write $\Chi^{*}$ for the dual lattice $\Hom_{\Z}(\Chi, \Z)$.

\begin{definition}
We call a submonoid $\wm$ of $\dw$ \textbf{smooth} if and only if there exists a smooth affine $G$-variety $X$ such that 
\begin{equation}
\C[X] \cong \oplus_{\lambda \in \wm} V(\lambda) \label{eq:smwm}
\end{equation}
as $G$-modules. 
\end{definition}

\begin{definition} \label{def:G-saturated}
Let  $\wm$ be a submonoid of the monoid $\dw$ of dominant weights of $G$.
We say that $\wm$ is \textbf{$G$-saturated} if the following
equality holds in $\wl$: 
\begin{equation}
\Z\wm \cap \dw = \wm.
\end{equation}
\end{definition}

\emph{For the remainder of this section, $\wm$ is a $G$-saturated submonoid of $\dw$.} Readers eager to check whether their favorite
$G$-saturated weight monoid $\wm$ is smooth can directly jump to Theorem~\ref{thm:main} and work their way backwards from there following the provided cross-references. 

As will be shown in Corollary~\ref{cor:ms_sat_irr}, if $\Gamma$ is $G$-saturated there is a unique affine $G$-variety for which the equality~(\ref{eq:smwm}) holds and which \emph{can be} smooth. We will denote it $X_{\wm}$. It is the ``most generic'' affine spherical $G$-variety with weight monoid $\wm$.  We recall that an irreducible (not necessarily affine or smooth) $G$-variety is \textbf{spherical} if it is normal and has an open $B$-orbit.

From $\wm$ we derive the following data:
\begin{enumerate}
\item the set of N-spherical roots of $X_{\wm}$, denoted by $\Sigma^{N}(\wm)$ and defined in Definition~\ref{def:Nadapted} (see also Proposition~\ref{prop:Nspherroots_Xwm}),
\item the valuation cone of $X_{\wm}$ , i.e.\ the set \begin{equation} \label{eq:valcone}
\V(\wm) = \{\nu \in \Hom_{\Z}(\Z\wm,\Q) \colon \<\nu,\osig\> \leq 0
\text{ for all } \osig \in \Sigma^N(\wm)\},
\end{equation}
\item a set of simple roots $S_\wm$.
\end{enumerate}

Proposition~\ref{prop:adapnsphroots} below tells us how to compute the set $\Sigma^{N}(\wm)$ from $\Gamma$, and $S_\wm$ is defined in Proposition~\ref{prop:localizroots}. We first introduce the relevant notions. 

\begin{definition} \label{def:scspherrootsG}
Let $\sigma$ be an element of the root lattice $\rl$ of $G$ and let $\sigma = \sum_{\alpha \in \sr} n_\alpha \alpha$ be its unique expression as a linear combination of the simple roots. The \textbf{support} of $\sigma$ is
\(\mathrm{supp}(\sigma)=\{\alpha \in \sr \colon n_\alpha \neq 0\}.\)
The \textbf{type} of $\mathrm{supp}(\sigma)$ is the Dynkin type of the subsystem generated by $\mathrm{supp}(\sigma)$ in the root system of $G$. The \textbf{set $\Sigma^{sc}(G)$ of spherically closed spherical roots of $G$} is the subset of  $\N \sr$ defined as follows: an element $\sigma$ of $\N\sr$ belongs to $\Sigma^{sc}(G)$ if 
after numbering the simple roots in $\mathrm{supp}(\sigma)$ like Bourbaki (see \cite{bourbaki-geadl47}) $\sigma$ is listed in Table~\ref{table:scspher}.      
\end{definition}

\begin{table}\caption{spherically closed spherical roots} \label{table:scspher}
\begin{center}
\begin{tabular}{ll}
Type of support & $\sigma$ \\
\hline
$\sf A_1$ & $\alpha$\\
$\sf A_1$ & $2\alpha$\\
$\mathsf A_1 \times \mathsf A_1$ & $\alpha+\alpha'$\\
$\mathsf A_n$, $n\geq 2$ & $\alpha_1+\ldots+\alpha_n$\\
$\mathsf A_3$ & $\alpha_1+2\alpha_2+\alpha_3$\\
$\mathsf B_n$, $n\geq 2$ & $\alpha_1+\ldots+\alpha_n$\\
                     & $2(\alpha_1+\ldots+\alpha_n)$\\
$\mathsf B_3$ & $\alpha_1+2\alpha_2+3\alpha_3$\\
$\mathsf C_n$, $n\geq 3$ & $\alpha_1+2(\alpha_2+\ldots+\alpha_{n-1})+\alpha_n$\\
$\mathsf D_n$, $n\geq 4$ & $2(\alpha_1+\ldots+\alpha_{n-2})+\alpha_{n-1}+\alpha_n$\\
$\mathsf F_4$ & $\alpha_1+2\alpha_2+3\alpha_3+2\alpha_4$\\
$\mathsf G_2$ & $4\alpha_1+2\alpha_2$\\          
& $\alpha_1+\alpha_2$
\end{tabular}
\end{center}
\end{table}
 
\begin{remark}
Note that $\Sigma^{sc}(G)$ is a finite set for every connected reductive group $G$. The notation $\Sigma^{sc}(G)$ is justified by Proposition~\ref{prop:Sigma_sc_G} below. 
\end{remark}

\begin{definition} \label{def:Sp}
Let $\wm$ be a set of dominant weights of $G$, that is $\wm \subset
\dw$. Then we define
\[S^p(\wm):= \{\alpha \in S \colon \<\lambda, \alpha^{\vee}\> = 0
\text{ for all $\lambda \in \wm$} \}. \]
\end{definition}
 
\begin{proposition} \label{prop:adapnsphroots}
Suppose $\wm$ is a $G$-saturated submonoid of $\dw$. If $\osig \in \Sigma^{sc}(G)$, then $\osig$ is an element of $\Sigma^N(\wm)$ if and only if the following conditions are all satisfied:
\begin{enumerate}[(i)]
\item $\osig$ is not a simple root; \label{item:osignotsimple}
\item $\osig \in \Z\wm$; \label{item:osiginlat}
\item $\osig$ is compatible with $S^p(\wm)$, that is: \label{compat}
\begin{itemize}
\item[-] if $\osig=\alpha_1+\ldots+\alpha_n$ with support of type $\mathsf{B}_n$ then
$\{\alpha_2, \alpha_3, \ldots, \alpha_{n-1}\} \inn S^p(\wm)$ and $\alpha_n \notin S^p(\wm)$;
\item[-] if $\osig=\alpha_1+2(\alpha_2+\ldots+\alpha_{n-1})+\alpha_n$ with support of type $\mathsf{C}_n$ then
$\{\alpha_3, \alpha_4, \ldots, \alpha_n\}
\inn S^p(\wm)$; 
\item[-] if $\osig$ is any other element of $\Sigma^{sc}(G)$ then 
\(\{\alpha \in \supp(\osig)\colon \<\alpha^{\vee}, \osig\> =0\}
\inn S^p(\wm)\);
\end{itemize}
\item if $\osig = 2 \alpha$  
then $\<\alpha^{\vee}, \gamma\> \in 2\Z$ for all $\gamma \in \Z\wm$; \label{item:corooteven}
\item if $\osig = \alpha+\beta$ with $\alpha, \beta \in S$ and $\alpha \perp \beta$, then $\<\alpha^{\vee}, \gamma\> = \<\beta^{\vee}, \gamma\>$ for all $\gamma \in \Z\wm$. \label{item:osigorthosum}
\end{enumerate} 
\end{proposition}

Proposition~\ref{prop:adapnsphroots} is a special case of \cite[Corollary 2.17]{msfwm}, as we will show on page~\pageref{proofpropadapn}. The proof of the following proposition is in Section~\ref{sec:smoothwm}.

\begin{proposition} \label{prop:localizroots} Let $\wm$ be a
  $G$-saturated submonoid of $\dw$. 
Among all the subsets $F$ of $S$ such that the relative interior of the cone spanned by $\{\alpha^{\vee}|_{\Z\wm} \colon \alpha \in F\}$ in $\Hom_{\Z}(\wm,\Q)$ intersects $\V(\wm)$ there is a unique one, denoted $S_{\wm}$, that contains all the others.  
\end{proposition}

\begin{remark}
\begin{enumerate}[(1)]
\item Note that $S^p(\wm) \subset S_{\wm}$ since
$\alpha^{\vee}|_{\Z\wm} =0$ for all $\alpha \in S^p(\wm)$. 
\item Determining $S_{\wm}$ is a finite (algorithmic) process. Indeed, $S \setminus S^p(\wm)$ is a finite set, and deciding for a given subset $F$ of $S\setminus S^p(\wm)$ whether the relative interior of the cone spanned by $\{\alpha^{\vee}|_{\Z\wm} \colon \alpha \in F\}$ intersects $\V(\wm)$ is equivalent to deciding whether a certain system of linear inequalities with integer coefficients has a solution in the positive rational numbers. The Fourier-Motzkin elimination algorithm does the latter in finitely many steps. 
\end{enumerate}
\end{remark}

Using $\Sigma^N(\wm)$ and $S_\wm$ we apply a smoothness criterion due to Camus (see~\cite[\S 6]{camus}) to decide whether $X_{\wm}$ is smooth. We recall and explain the criterion in Section~\ref{sec:camus}. To state the theorem, we need one more definition.  

\begin{definition} \label{def:admissible}
Let $\sr$ be the set of simple roots of a root system. Let $\sr^p$ be a subset $\sr$.  Let $\Sigma^N$ be a subset $\N\sr$. We say that the triple $(\sr, \sr^p, \Sigma^N)$ is \textbf{admissible} if there exists a finite set $I$ and for every $i \in I$ a triple $(\sr_i, \sr^p_i, \Sigma_i)$ from List~\ref{list:pat} below and an automorphism $f_i$ of the Dynkin diagram of $\sr_i$ such that the Dynkin diagram of $S$ is the union over $i\in I$ of the Dynkin diagrams of the $\sr_i$, that $\sr^p = \cup_if_i(\sr^p_i)$ and that $\Sigma^N = \cup_i f_i(\Sigma_i)$. 
\end{definition}

\begin{List}[Primitive admissible triples] \label{list:pat}
\par
\begin{enumerate}[1.]
\item $(\sr,\sr,\emptyset)$ where $\sr$ is the set of simple roots
  of an irreducible root system;
\item $(\sA_n, \{\alpha_2,\alpha_3, \ldots,\alpha_{n}\}, \emptyset)$
  for $n \ge 1$;
\item $(\sA_n, \{\alpha_1, \alpha_3, \alpha_5, \ldots, \alpha_{n-1}\},\{\alpha_1+2\alpha_2+\alpha_3, \alpha_3+ 2\alpha_4 + \alpha_5, \ldots, \alpha_{n-3}+2\alpha_{n-2}+\alpha_{n-1}\})$ for $n \ge 4$, $n$ even;
\item $(\sA_n \times \sA_k, \{\alpha_{k+2}, \alpha_{k+3}, \ldots, \alpha_n\}, \{\alpha_1+\alpha_1', \alpha_2+\alpha_2', \ldots, \alpha_k + \alpha_k'\})$ for $n > k \ge 2$;\label{item:admiss:AnAk}
\item $(\sC_n, \{\alpha_2,\alpha_3, \ldots,\alpha_{n}\}, \emptyset)$  for $n \ge 2$;
\item $(\sD_5, \{\alpha_2, \alpha_3, \alpha_4\}, \{\alpha_2 + 2\alpha_3 + \alpha_4 + 2\alpha_5\})$.
\end{enumerate}
\end{List}

\begin{remark} \label{rem:emptytriple}
Note that  we allow $I=\emptyset$ in
Definition~\ref{def:admissible}: the triple $(\emptyset, \emptyset, \emptyset)$ is
admissible. 
\end{remark}

Here is the main result of this paper (Theorem~\ref{thm:general}), specialized to the  case of $G$-saturated weight monoids. The proof is given on page~\pageref{proof_thm_main} in Section~\ref{sec:smoothwm}.

\begin{theorem}  \label{thm:main}
Let $\wm$ be a $G$-saturated monoid of dominant weights of
$G$. Let $S^p(\wm)$ be the set of simple roots that are orthogonal to
$\wm$ as in Definition~\ref{def:Sp}, $S_{\wm}$ the set  of simple roots as in Proposition~\ref{prop:localizroots} and $\Sigma^N(\wm)$ the set given by Proposition~\ref{prop:adapnsphroots}. Then $\wm$ is the weight monoid of a smooth affine spherical $G$-variety if and only if 
\begin{enumerate}[(a)]
\item $\{\alpha^{\vee}|_{\Z\wm}\colon \alpha \in S_{\wm}\setminus S^p(\wm)\}$ is a
  subset of a basis of $(\Z\wm)^*$; and \label{cond:partofbasis}
\item for all $\alpha,\beta \in S_{\wm}\setminus S^p(\wm)$ such that $\alpha\neq\beta$ and $\alpha^{\vee}|_{\Z\wm} = \beta^{\vee}|_{\Z\wm}$ we have $\alpha+\beta \in \Z\wm$; and \label{cond:sumisNsphericalroot}
\item the triple $(S_{\wm},S^p(\wm), \Sigma^N(\wm) \cap \Z S_{\wm})$ is admissible (see Definition~\ref{def:admissible}). 
\end{enumerate}
\end{theorem}

\begin{remark}
Our proof of Theorem~\ref{thm:main} relies on the classification of wonderful varieties by spherical systems, which was conjectured in \cite{luna-typeA} and was known as the Luna Conjecture. It is proved in \cite{bravi-pezzini-primwonderful}. Specifically, we use the classification in Proposition~\ref{prop:luna-adapt}; see also Remark~\ref{rem:thm:general}(\ref{rem:thm:general:lc}). Another proof of the Luna Conjecture has been proposed in \cite{cupit-wvgr-prep}. 
\end{remark}

\begin{remark} 
\begin{enumerate}[(a)]
\item Several of the statements in Sections \ref{subsec:spher_roots_generic_wiht_wm} -- \ref{subsec:Gsat} appeared in \cite{avdeev&cupit-irrcomps-arxivv2}, cf.~Remark~\ref{rem:acf}(\ref{acf_comp}). However, the present
paper is independent of \loccit\  
The content of Sections~\ref{subsec:spherroots_adapted_weightmonoid} and \ref{subsec:Gsat} was inspired by an unpublished working document of Luna's from 2005.
\item In Section~\ref{sec:camus} we present Camus's smoothness criterion, with a complete exposition of its original proof following~\cite{camus}. Gagliardi has published another proof of the criterion in the paper \cite{gagliardi-camus}, which also includes the so-called ``Luna diagrams'' of the spherical modules.
\end{enumerate}

\end{remark}

As an application we study affine model varieties of simply connected semisimple groups.  Recall that a quasi-affine $G$-variety $Y$ is called a \textbf{model} $G$-variety if the $G$-module $\C[Y]$ contains every irreducible representation of $G$ with multiplicity one. This is equivalent to being a quasi-affine spherical $G$-variety with weight monoid equal to $\dw$. Homogeneous model $G$-varieties were first introduced in~\cite{bernstein-gelfand-gelfand-model} and then further studied in  \cite{gelfand-zelevinski-model1, gelfand-zelevinski-model2, adams-huang-vogan, luna-model}.

When $G$ is simply connected and semisimple, $\dw$ is free and $G$-saturated. Applying Theorem~\ref{thm:main} to the monoid $\dw$, we obtain the following theorem. Its proof is given in Section~\ref{sec:model}.

\begin{theorem} \label{thm:model}
Let $G$ be a simply connected semisimple linear algebraic group. There
exists a smooth affine model $G$-variety if and only if the simple factors of $G$ are of type $\sA$  or of type $\sC$. 
\end{theorem}

The ``if'' part of Theorem~\ref{thm:model} is not new, see Example~\ref{ex:model} below.

\begin{example}\label{ex:model}
For all $n,k \in \Z_{> 0}$ the groups $\SL(n)$, $\Sp(2n)$ and $\SL(k) \times \Sp(2n)$ have smooth
affine model varieties, while e.g.\ $\Spin(k)$ does not. The existence of such varieties for $\SL(n)$ and $\Sp(2n)$ was already known: the variety for the former group with $n>1$ odd is $\SL(n)/\Sp(n-1)$, and for $n$ even it is the homogeneous vector bundle $\SL(n)\times^{\Sp(n)} \C^n$, while the variety for $\Sp(2n)$ is the homogeneous vector bundle $\Sp(2n)\times^{\Sp(2a)\times\Sp(2b)}\C^{2b}$ where $a=b=n/2$ if $n$ is even, and $a=b-1=(n-1)/2$ if $n$ is odd. In order to decide the existence of a smooth affine model variety in general, one can apply Theorem~\ref{thm:main} for $\wm = \dw$ when $G$ is any connected reductive group. For example, this way one can recover that the group $\SO(2n + 1)$, with $n\geq 1$, has a smooth affine model variety. In fact, it is $\SO(2n + 1)/ \GL(n)$ (cf.\ \cite{luna-model}).
\end{example}

\subsection*{Acknowledgement}
The authors started this project at the Institut Fourier in the summer of 2011, and thank the institute and Michel Brion for hosting them. They also thank Domingo Luna for fruitful discussions and suggestions, and for sharing his 2005 working document in which several of the ideas used in Section~\ref{sec:combinatorics} were outlined. 
The first named author was partially supported by the DFG Schwerpunktprogramm 1388 -- Darstellungstheorie.
The second named author thanks Friedrich Knop and the Emmy Noether Zentrum for hosting him in the summers of 2012, 2013, 2014 and 2015. He received support from the City University of New York PSC-CUNY Research Award Program and from the National Science Foundation through grant DMS 1407394.

\section{Spherical combinatorics and Alexeev and Brion's moduli scheme} \label{sec:combinatorics}

\subsection{Alexeev and Brion's moduli scheme}
We recall that the weight monoid $\Gamma$ of an affine spherical variety $X$ is finitely generated, because $\CC[X]$ is a finitely generated ring. Since $X$ is normal, its weight monoid satisfies the following equality in $\wl \otimes_{\Z}\Q$:
\begin{equation}
\Z\wm \cap \Q_{\ge 0}\wm = \wm. \label{eq:3}
\end{equation}
We will call a submonoid $\wm$ of $\dw$ satisfying conditition~(\ref{eq:3}) \textbf{normal}. Notice that any such $\wm$ is finitely generated by Gordan's lemma.

Let now $\wm$ be a normal submonoid of $\dw$. In \cite{alexeev&brion-modaff}, Alexeev and Brion introduced a moduli scheme $\ms$, which parametrizes affine spherical $G$-varieties with weight monoid $\wm$. To describe $\ms$ more precisely, put $V(\wm) := \oplus_{\lambda \in \wm} V(\lambda)$ and equip $V(\wm)^U$ with a $T$-multiplication law by choosing an isomorphism $V(\wm)^U \cong \C[\wm]$. The (closed) points of $\ms$ are in one-to-one correspondence with the $G$-multiplication laws on $V(\wm)$ that extend the multiplication law on $V(\wm)^U$. Alexeev and Brion showed that \emph{$\ms$ is an affine scheme of finite type over $\C$ which represents the functor $\mathsf{Schemes} \to \mathsf{Sets}$ that associates with
any scheme $Z$, the set of families of algebra structures of type $\wm$ over $Z$}. For an introduction to this moduli scheme, we refer to \cite[Section 4.3]{brion-ihs}. 

Thanks to $\ms$, we can make precise the notion of ``generic'' affine spherical $G$-variety with weight monoid $\wm$. Alexeev and Brion equipped $\ms$ with an action of the maximal torus $T$ of $G$ and showed that \emph{there is a natural bijection between the the $T$-orbits on $\ms$ and the isomorphism classes of affine spherical $G$-varieties with weight monoid $\wm$}, see \cite[Theorem 1.12]{alexeev&brion-modaff}. When $X$ is an affine spherical variety with weight monoid $\wm$, we will write $T\cdot [X]$ for the $T$-orbit on $\ms$ corresponding to the $G$-isomorphism class of $X$. 

\begin{definition}
Let $X$ be an affine spherical $G$-variety with weight monoid $\wm$. We will say that $X$ is \textbf{generic} if $T\cdot[X]$ is an open subset of $\ms$. 
\end{definition}

\begin{proposition}[{\cite[Corollary 2.9]{alexeev&brion-modaff}}] \label{prop:smoothgeneric}
If $X$ is a smooth affine spherical $G$-variety with weight monoid $\wm$, then $X$ is generic.
\end{proposition}

Thanks to \cite[Corollary 3.4]{alexeev&brion-modaff} we know that that there are, up to isomorphism, only finitely many affine spherical varieties with a given weight monoid $\wm$. Equivalently, \emph{$\ms$ contains only finitely many $T$-orbits}. This implies the following proposition.

\begin{proposition} \label{prop:denseTorb_irrcomp}
Let $\wm$ be a normal submonoid of $\dw$. Every irreducible component of $\ms$ contains a (unique) open (and dense) $T$-orbit. Equivalently, for every irreducible component $Z$ of $\ms$ there exists a unique $T$-orbit $T\cdot [X]$ on $\ms$ such that $Z = \overline{T\cdot [X]}$, where $Z$ is equipped with its reduced induced scheme structure and $\overline{T\cdot [X]}$ is the closure of $T\cdot[X]$. 
\end{proposition}  

It follows from Proposition~\ref{prop:denseTorb_irrcomp} that the map $X \mapsto \overline{T\cdot[X]}$ yields a bijection between isomorphism classes of generic varieties with weight monoid $\wm$ and irreducible components of $\ms$.

Here is another result from \cite{alexeev&brion-modaff} we will need. It establishes a crucial link between the geometry of $\ms$ and a combinatorial invariant of the varieties $\ms$ parametrizes: their root monoids. We recall that the \textbf{root monoid} $\mathscr M_X$ of a quasi-affine $G$-variety $X$ is the
submonoid of $\wl$ generated by 
\[\{\lambda+\mu-\nu \mid \lambda, \mu,
\nu \in \dw \text{ such that } \C[X]_{(\nu)} \cap
(\C[X]_{(\lambda)}\C[X]_{(\mu)}) \neq 0\},\]
where for $\gamma \in \dw$ we used $\C[X]_{(\gamma)}$ for the isotypic
component of type $\gamma$ in $\C[X]$ and
$\C[X]_{(\lambda)}\C[X]_{(\mu)}$ is the subspace of $\C[X]$
spanned by the set $\{fg \colon f \in \C[X]_{(\lambda)}, g \in
\C[X]_{(\mu)}\}$.   By \cite[Theorem 1.3]{knop-auto}, the saturation of $\mathscr M_X$, that is, the intersection of the cone spanned by $\mathscr M_X$ and the group generated by $\mathscr M_X$, is a free submonoid of $\rl$ (cf.~Remark~\ref{rem:acf}).

\begin{proposition}[{\cite[Proposition 2.13]{alexeev&brion-modaff}}] \label{prop:monTorbclos}
If $X$ is an affine spherical $G$-variety with weight monoid $\wm$, then the closure $\overline{T\cdot[X]}$ of the corresponding $T$-orbit on $\ms$ is $T$-isomorphic to $\Spec \C[\mathscr M_X]$, where $\mathscr M_X$ is the root monoid of $X$.
\end{proposition}

\begin{remark} \label{rem:acf}
\begin{enumerate}[(a)]
\item In their recent preprints \cite{avdeev&cupit-irrcomps-arxivv2,avdeev&cupit-newold-arxivv2} Avdeev and Cupit-Foutou have proposed a proof of Knop's conjecture that the root monoid $\mathscr M_X$ of an affine spherical variety $X$ is free, and of Brion's conjecture in \cite{brion-ihs} that the irreducible components of $\ms$, equipped with their reduced induced scheme structure, are affine spaces, under the assumption that $\wm$ is normal. 
\item \label{acf_comp} Up to technicalities, Proposition~\ref{prop:orbincl}(\ref{orbincl_item1}), Proposition~\ref{prop:orbincl}(\ref{orbincl_item2}), Corollary~\ref{cor:genmaxadap}, Corollary~\ref{cor:msirreducible}, Proposition~\ref{prop:luna-adapt} and Corollary~\ref{cor:ms_sat_irr}(\ref{cor:ms_sat_irr_a}) below are the same statements as Corollary~2.8(b), Proposition~3.14, Theorem~4.9, Corollary~4.10, Theorem~4.8 and Theorem~6.5, respectively, in~\cite{avdeev&cupit-irrcomps-arxivv2} (for the last one, this is true taking \cite[Corollary~6.3]{avdeev&cupit-irrcomps-arxivv2} into account).
\end{enumerate}
\end{remark}

\subsection{Luna invariants and spherical closure}
In this section, we first recall some basic notions in the theory of spherical varieties.
For more details we refer to
\cite{knop-lv, luna-typeA}. We then state Proposition~\ref{prop:doubling_sr}, which is essentially due to Losev and which describes the relationship between the three standard normalizations of the so-called `spherical roots' of a spherical variety.  

Let $X$ be a spherical $G$-variety with open orbit $G/H$. The basic invariants the theory of spherical varieties associates to $X$ are defined as follows.  
\begin{enumerate}[1.]
\item The \textbf{lattice} of
$X$, denoted $\lat(X)$, is the subgroup of $\wl$ consisting of the $B$-weights of
$B$-eigenvectors in the field of rational functions $\C(X)$.
\item Let $\nu\colon \C(X)^{\times}\to\Q$ be a discrete valuation. Then $\nu$ induces an element of $\Hom_\ZZ(\lat(X),\Q)$, denoted $\rho_X(\nu)$, by
\[
\langle \rho_X(\nu),\gamma\rangle = \nu(f_\gamma)
\]
where $f_\gamma\in\C(X)$ is a $B$-eigenvector of $B$-weight
$\gamma\in\lat(X)$. If $D\subset X$ is a prime divisor, we denote by
$\nu_D$ the associated discrete valuation, and for simplicity by
$\rho_X(D)$ the element $\rho_X(\nu_D)$ of $\Hom_\Z(\lat(X),\Q)$.
\item A \textbf{color} of $X$ is a $B$-stable but not $G$-stable
  prime divisor of $X$. The set of colors of $X$ is denoted
  $\col(X)$. 
\item The \textbf{Cartan pairing} of $X$ is the bilinear map
\[
c_X\colon \Z\col(X) \times\lat(X) \to \Z
\]
given by extending by linearity the elements $\rho_X(D) \in \Hom_\Z(\lat(X),\Q)$ with
$D\in\col(X)$. In particular, for $D \in \col(X)$ and $\gamma \in \lat(X)$,
\[c_X(D,\gamma) = \<\rho_X(D), \gamma\>.\]

\item Let $P_X$ be the stabilizer of the open $B$-orbit of $X$ and denote by $S^p(X)$ the subset of simple
roots corresponding to $P_X$, which is a parabolic subgroup of $G$ containing $B$.

\item We use $\V(X)$ for the set of $G$-invariant $\Q$-valued
  discrete valuations of $\C(X)$ and identify $\V(X)$ with its image
  in $\Hom_\Z(\lat(X),\Q)$ via the map $\rho_{X}$. We call $\V(X)$ the
  \textbf{valuation cone} of $X$. 

\item By \cite{brion-gensymm}, $\V(X)$ is a co-simplicial cone. The set of \textbf{spherical roots} $\Sigma(X)$ of $X$ is the minimal set of primitive elements of $\lat(X)$ such that \label{item:6}
\[
\V(X) = \left\{ \eta\in \Hom_\Z(\lat(X),\Q) \;\middle\vert\; \langle\eta,\sigma\rangle \leq 0 \;\forall\sigma\in \Sigma(X) \right\}.
\]
\item A color $D$ of $X$ is \textbf{moved} by a simple root $\alpha\in S$ if $D$ is not stable under the minimal parabolic subgroup of $G$ containing $B$ and associated with the simple root $\alpha$. If $\alpha\in S\cap \Sigma(X)$ then we denote by $\A(X,\alpha)$ the set of colors of $X$ moved by $\alpha$, and we set
\[
\A(X) := \bigcup_{\alpha\in S\cap \Sigma(X)} \A(X,\alpha).
\]
\end{enumerate}
Note that all of these invariants are equivariantly birational: they only depend on $G/H$. We will also need two other sets of spherical roots associated to $X$, namely $\Sigma^{sc}(X)$ and $\Sigma^N(X)$. Like $\Sigma(X)$, they consist of normal vectors to  $\V(X)$, but possibly of other lengths. Before defining them, we recall Luna's notion of spherical closure~\cite[\S 6.1]{luna-typeA}. Recall that subgroup $H$ of $G$ is called a \textbf{spherical subgroup} if $G/H$ is a spherical $G$-variety. Then the quotient $N_G(H)/H$ naturally acts on $G/H$ by $G$-equivariant automorphisms, inducing an action of $N_G(H)$ on the set of colors of $G/H$.  The kernel of this last action is the \textbf{spherical closure} of $H$, denoted by $\overline H$. If $H=\overline H$, then $H$ is said to be \textbf{spherically closed}. In general, the normalizer of $\overline H$ may be bigger that the normalizer of $H$, but $\overline H$ is always spherically closed, that is
\begin{equation}
\overline{\overline{H}} = \overline{H}.
\end{equation}
This follows from \cite[Lemma~2.4.2]{bravi-luna-f4}; see \cite[Proposition~3.1]{pezzini-redaut} for a direct proof.

\begin{definition}
Let $X$ be a spherical $G$-variety with open orbit $G/H$. We define
\begin{align}
\Sigma^{sc}(X)&:=\Sigma^{sc}(G/H):= \Sigma(G/\overline{H})\\
\Sigma^{N}(X)&:=\Sigma^{N}(G/H):= \Sigma(G/N_G(H)). \label{eq:2}
\end{align}
The latter is called the set of \textbf{N-spherical roots} of $X$.
\end{definition}
It follows from~\cite[Theorem 1.3]{knop-auto} that when $X$ is quasi-affine, the set  $\Sigma^{N}(X)$ defined in equation~(\ref{eq:2}) is the basis of the saturation of $\mathscr M_X$. Thanks to \cite{losev-uniqueness}, the relation between $\Sigma(X)$, $\Sigma^{sc}(X)$ and $\Sigma^{N}(X)$ is well understood: the three sets have the same cardinality, and for every $\sigma \in \Sigma(X)$, either $\sigma$ or its double belongs to $\Sigma^{sc}(X)$; and similarly for $\Sigma^N(X)$. In the next proposition we precisely say which elements of $\Sigma(X)$ have to be doubled.

\begin{proposition}[Losev] \label{prop:doubling_sr}
Let $X$ be a spherical variety. Then $\Sigma^N(X)$ is obtained from $\Sigma(X)$ by replacing $\sigma$ with $2\sigma$ for all $\sigma$ satisfying any one of the following conditions:
\begin{enumerate}
\item\label{prop:doubling_sr:A} $\sigma\in\Sigma\cap \sr$ with $\rho_X(D_\sigma^+)=\rho_X(D_\sigma^-)$ where $\{D_\sigma^+,D_{\sigma}^-\}=\A(X,\sigma)$,
\item\label{prop:doubling_sr:B}  $\sigma = \alpha_1+\ldots+\alpha_n$, where $\{\alpha_1,\ldots,\alpha_n\}\subseteq S$ has type $\sB_n$ and $\alpha_i\in S^p$ for all $i\in\{2,\ldots,n\}$,
\item\label{prop:doubling_sr:G} $\sigma = 2\alpha_1+\alpha_2$, where $\{\alpha_1,\alpha_2\}\subseteq \sr$ has type $\sG_2$,
\item\label{prop:doubling_sr:root} $\sigma$ is not in the root lattice of $G$.
\end{enumerate}
The set $\Sigma^{sc}(X)$ is obtained from $\Sigma(X)$ by replacing $\sigma$ with $2\sigma$ for all $\sigma$ satisfying satisfying condition (\ref{prop:doubling_sr:B}), (\ref{prop:doubling_sr:G}), or (\ref{prop:doubling_sr:root}).
\end{proposition}
\begin{proof}
The statement about $\Sigma^N(X)$ is exactly \cite[Theorem 2]{losev-uniqueness}, and the statement about $\Sigma^{sc}(X)$ is exactly \cite[Lemme 7.1]{luna-typeA}.

For the second statement, we point out that the proof of \loccit\ uses the general classification of spherical homogeneous spaces. A similar but more self-contained argument, essentially relying only on \cite[Theorem 2]{losev-uniqueness}, goes as follows.

Let $G/H$ be the open orbit of $X$ and let $\Sigma'$ be equal to $\Sigma(X)$ where we replace $\sigma$ with $2\sigma$ for all $\sigma$ satisfying conditions (\ref{prop:doubling_sr:B}), (\ref{prop:doubling_sr:G}), and (\ref{prop:doubling_sr:root}). By \cite[Theorem 2]{losev-uniqueness} we have $\Sigma'\subseteq \Sigma^N(X)$, and we have
\[
\frac{\ZZ\Sigma'}{\ZZ\Sigma^N(X)}\subseteq \frac{\Lambda(X)}{\ZZ\Sigma^N(X)}=\frac{\Lambda(X)}{\Lambda(G/N_GH)} \cong \frac{N_GH}{H}
\]
where the equality is \cite[Corollary 6.5]{knop-auto} and the isomorphism is \cite[Lemma 2.4]{gandini-spherorbclos}. Then $\ZZ\Sigma'$ corresponds to a subgroup $K$ of $N_GH$ containing $H$, such that $\Lambda(G/K)=\ZZ\Sigma'$. From \cite[Theorem 4.4 and proof of Theorem 6.1]{knop-lv} we deduce that $\Sigma(G/K)=\Sigma'$.

According to \cite[Section 2.3]{luna-typeA} the number of colors of $G/H$ and of $G/K$ are equal. In other words the natural map $G/H\to G/K$ induces a bijection between the sets of colors of $G/H$ and $G/K$, whence $K\subseteq \overline H$. Applying once again \cite[Theorem 4.4 and proof of Theorem 6.1]{knop-lv} we have that $\Sigma'$ and $\Sigma^{sc}(X)$ are equal up to replacing some elements of the first set with positive rational multiples, and thanks to the classification of spherical roots we have that the coefficients can only be $1$ or $2$, i.e.\ in particular $\Sigma^{sc}(X)\subseteq \Sigma'\cup 2\Sigma'$.

Assume that there exists an element in $\Sigma^{sc}(X)$ not in $\Sigma'$. It has the form $2\sigma$ with $\sigma\in \Sigma'$. Then $\sigma\in \Sigma(X)$ and $2\sigma\in \Sigma^N(X)$, and by definition of $\Sigma'$ the only possibility is that $\sigma$ satisfies condition (\ref{prop:doubling_sr:A}), i.e.\ $\sigma\in S\cap \Sigma(G/H)$. But in this case $G/H$ and $G/\overline H$ would have a different number of colors moved by $\sigma$ (resp.\ $2$ and $1$), which is impossible.

Therefore $\Sigma^{sc}(X)\subseteq \Sigma'$, and since the two sets have the same finite cardinality, they are equal.
\end{proof}

The notation $\Sigma^{sc}(G)$ of Definition~\ref{def:scspherrootsG} is justified by the following.
\begin{proposition}[see {\cite[\S 1.2]{luna-typeA}}] \label{prop:Sigma_sc_G}
An element $\sigma$ of $\N\sr$ belongs to $\Sigma^{sc}(G)$ if and only if there exists a spherically closed spherical subgroup $K$ of $G$ such that
$\Sigma(G/K) = \{\sigma\}$. 
\end{proposition}

Given a spherical $G$-variety $X$, the triple
\[\s(X) = (S^p(X), \Sigma^{sc}(X), \A(X)),\]
equipped with the restriction of the Cartan pairing $c_X$ to $\Z\A(X) \times \Z\Sigma^{sc}(X)$, is a spherically closed spherical system in the following sense.

\begin{definition} \label{def:spherical-systems}
Let $(S^p,\Sigma,\mathbf A)$ be a triple where $S^p$ is a  subset of $\sr$,
$\Sigma$ is a subset of $\Sigma^{sc}(G)$
and $\mathbf A$ is a finite set endowed with a $\Z$-bilinear pairing $c\colon \Z\mathbf A\times\Z\Sigma\to\Z$. 
For every $\alpha \in \Sigma \cap S$, let $\mathbf A (\alpha)$ denote the set $\{D \in \mathbf A : c(D,\alpha)=1 \}$. 
Such a triple is called a \textbf{spherically closed spherical $G$-system} if all the
following axioms hold: 
\begin{itemize} 
\item[(A1)] for every $D \in \mathbf A$ and every $\sigma \in \Sigma$,
  we have that $c(D,\sigma)\leq 1$ and that
  if $c(D, \sigma)=1$ then $\sigma \in S$; 
\item[(A2)] for every $\alpha \in \Sigma \cap S$, $\mathbf A(\alpha)$
  contains exactly two elements, which we denote by $D_\alpha^+$ and
  $D_\alpha^-$, and for all $\sigma \in \Sigma$ we have $c(D_\alpha^+,\sigma) + c(D_\alpha^-,\sigma) = \langle \alpha^\vee , \sigma \rangle$; 
\item[(A3)] the set $\mathbf A$ is the union of $\mathbf A(\alpha)$ for all $\alpha\in\Sigma \cap S$; 
\item[($\Sigma 1$)] if $2\alpha \in \Sigma \cap 2S$ then $\frac{1}{2}\langle\alpha^\vee, \sigma \rangle$ is a non-positive integer for all $\sigma \in \Sigma \setminus \{ 2\alpha \}$; 
\item[($\Sigma 2$)] if $\alpha, \beta \in S$ are orthogonal and $\alpha + \beta$ belongs to $\Sigma$ then $\langle \alpha ^\vee , \sigma \rangle = \langle \beta ^\vee , \sigma \rangle$ for all $\sigma \in \Sigma$; 
\item[(S)] every $\sigma \in \Sigma$ is \textbf{compatible} with
  $S^p$, that is, for every $\sigma \in \Sigma$ there exists a
  spherically closed spherical subgroup $K$ of $G$ with
  $S^p(G/K)=S^p$ and $\Sigma(G/K)=\{\sigma\}$.
\end{itemize}
\end{definition}

\begin{remark}\label{rem:spherical-systems} 
\begin{enumerate}[1.]
\item Condition (S) of Definition~\ref{def:spherical-systems} can be stated
in purely combinatorial terms as follows (see \cite[\S 1.1.6]{bravi-luna-f4}).
A spherically closed spherical root $\sigma$ is compatible with $S^p$ if and only if:
\begin{itemize}
\item[-] in case $\sigma=\alpha_1+\ldots+\alpha_n$ with support of type $\sB_n$
\[\{\alpha\in\supp(\sigma)\colon\<\alpha^\vee,\sigma\>=0\}\setminus\{\alpha_n\}
\subseteq S^p\subseteq\{\alpha\in S\colon\<\alpha^\vee,\sigma\>=0\}\setminus\{\alpha_n\},\]
\item[-] in case $\sigma=\alpha_1+2(\alpha_2+\ldots+\alpha_{n-1})+\alpha_n$ with support of type $\sC_n$
\[\{\alpha\in\supp(\sigma)\colon\<\alpha^\vee,\sigma\>=0\}\setminus\{\alpha_1\}
\subseteq S^p\subseteq\{\alpha\in S\colon\<\alpha^\vee,\sigma\>=0\},\] 
\item[-] in the other cases
\[\{\alpha\in\supp(\sigma)\colon\<\alpha^\vee,\sigma\>=0\}
\subseteq S^p\subseteq\{\alpha\in S\colon\<\alpha^\vee,\sigma\>=0\}.\]
\end{itemize} 
\item Definition~\ref{def:spherical-systems} is the same as \cite[Definition 2.5]{msfwm}. It combines the standard
  definition of spherical system, see \cite[\S 2]{luna-typeA}, with the
  requirement that it be spherically closed, see \cite[\S 7.1]{luna-typeA}
  and \cite[\S 2.4]{bravi-luna-f4}.
\end{enumerate}
\end{remark}

\subsection{Spherical roots of a generic affine spherical variety with weight monoid $\wm$.} \label{subsec:spher_roots_generic_wiht_wm}
In this section, we recall from \cite{msfwm} the definition of spherical roots that are `adapted' to a given normal submonoid $\wm$ of $\dw$. We deduce that the generic affine spherical varieties $X$ with weight monoid $\wm$ are those for which $\Sigma^{sc}(X)$ is a maximal set of spherical roots adapted to $\wm$; see Corollary~\ref{cor:genmaxadap}.

\begin{definition} \label{def:adapted}
Let $\wm$ be a normal submonoid of $\dw$. We say that a subset $\Sigma$ of $\Sigma^{sc}(G)$ is \textbf{adapted to $\wm$} if there exists an affine spherical variety $X$ such that $\wm(X)=\wm$ and $\Sigma^{sc}(X)=\Sigma$. We say that an element $\sigma$ of $\Sigma^{sc}(G)$ is adapted to $\wm$ 
if $\{\sigma\}$ is adapted to $\wm$. We use $\Sigma^{sc}(\wm)$ for the set of all $\sigma \in \Sigma^{sc}(G)$ that are adapted to $\wm$.  
\end{definition}
\begin{remark}
In general, $\Sigma^{sc}(\wm)$ is not adapted to $\wm$. The following example is due to Luna: when $G = \SL(2) \times \SL(2)$ and $\wm = \N\{2\omega, 4\omega+2\omega'\}$, then $\Sigma^{sc}(\wm) = \{\alpha, 2\alpha'\}$ and one checks that this set is not adapted to $\wm$. 
\end{remark}

\begin{proposition} \label{prop:orbincl} 
Let $X$ be an affine spherical $G$-variety with weight monoid $\wm$. 
\begin{enumerate}[(a)]
\item If $Y$ is also an affine spherical $G$-variety with weight monoid $\wm$, then $T\cdot[Y] \inn \overline{T\cdot[X]}$ as subsets of $\ms$ if and only if $\Sigma^{sc}(Y) \inn \Sigma^{sc}(X)$. \label{orbincl_item1}
\item For every subset $\Sigma'$ of $\Sigma^{sc}(X)$ there exists an affine $G$-spherical variety $Y'$ with weight monoid $\wm$ and $\Sigma^{sc}(Y') = \Sigma'$.  \label{orbincl_item2}
\item $\Sigma^{sc}(X) \inn \Sigma^{sc}(\wm)$. \label{orbincl_item3}
\end{enumerate}
\end{proposition}
\begin{proof}
We begin with assertion (\ref{orbincl_item1}). As recalled in Proposition~\ref{prop:monTorbclos}, Alexeev and Brion showed that $\overline{T\cdot[X]} = \Spec \C[\mathscr{M}_X]$. A basic fact in the theory of (not necessarily normal) affine toric varieties is that we have the following inclusion preserving one-to-one correspondence \cite[Theorem 3.A.3]{cox_little_schenck-toricbook}:
\begin{align} \label{eq:faceorbcorr}
\{\text{faces of $\Q_{\ge 0} \mathscr{M}_X$}\} &\to \{\text{orbit closures in $\overline{T\cdot[X]}$}\}\\
\mathcal{F} &\mapsto \Spec \C[\mathscr{M}_X \cap \mathcal{F}]
\end{align} 
Consequently, if $T\cdot [Y] \inn \overline{T\cdot[X]}$, or equivalently,  if $\overline{T\cdot[Y]} \inn  \overline{T\cdot[X]}$ then 
\begin{equation} \label{eq:rootmonintface}
\mathscr{M}_Y = \mathscr{M}_X \cap \mathcal{F} \text{ for some face $\mathcal{F}$ of  $\Q_{\ge 0} \mathscr{M}_X$}.
\end{equation}
By \cite[Theorem 1.3]{knop-auto}, the equality \eqref{eq:rootmonintface} holds if and only if $\Sigma^N(Y) \inn \Sigma^N(X)$ up to multiples. Recall that the elements of $\Sigma^N(X)$ are integer multiples of the elements of $\Sigma(X)$, and that the elements of $\Sigma(X)$ are primitive in the lattice $\lat(X)$. Since $\lat(X) = \lat(Y) = \Z\wm$, it follows that \eqref{eq:rootmonintface} is equivalent to $\Sigma(Y) \inn \Sigma(X)$. By Proposition~\ref{prop:doubling_sr}, the latter inclusion holds if and only if $\Sigma^{sc}(Y) \inn \Sigma^{sc}(X)$. This proves the ``only if'' statement of assertion (\ref{orbincl_item1}).   
We turn to the converse. It follows from assertion (\ref{orbincl_item2}) that if $\Sigma^{sc}(Y) \inn \Sigma^{sc}(X)$ then there exists an affine spherical $G$-variety $\widetilde{Y}$ with weight monoid $\wm$, and $\Sigma^{sc}(\widetilde{Y}) = \Sigma^{sc}(Y)$. It follows from \cite[Theorem 1.2]{losev-knopconj} that $\widetilde{Y}$ is $G$-isomorphic to $Y$. An application of \cite[Theorem 1.12]{alexeev&brion-modaff} finishes the proof of assertion (\ref{orbincl_item1}). 

Assertion (\ref{orbincl_item2}) is a formal consequence of the correspondence \eqref{eq:faceorbcorr} and \cite[Theorem 1.3]{knop-auto}. Indeed, $\Q_{\ge 0} \Sigma'$ is a face of $\Q_{\ge 0} \Sigma^{sc}(X)=\Q_{\ge 0} \mathscr{M}_X$ and so corresponds to a $T$-orbit closure $\overline{T\cdot[Y']}$ in $\overline{T\cdot[X]}$. Applying \cite[Theorem 1.3]{knop-auto} as above, it follows that $\Sigma^{sc}(Y') = \Sigma'$. 

Finally, assertion (\ref{orbincl_item3}) follows by applying (\ref{orbincl_item2}) to the singletons in $\Sigma^{sc}(X)$. 
\end{proof}

\begin{corollary} \label{cor:genmaxadap}
If $X$ is an affine spherical $G$-variety with weight monoid $\wm$, then $X$ is generic if and only if there is no subset of $\Sigma^{sc}(G)$ that strictly contains $\Sigma^{sc}(X)$ and is adapted to $\wm$.
\end{corollary}
\begin{proof}
This is a formal consequence of Propositions \ref{prop:orbincl} and \ref{prop:denseTorb_irrcomp} above and \cite[Theorem 1.12]{alexeev&brion-modaff}, as we now explain. 
We first assume that $X$ is generic, and show that $\Sigma^{sc}(X)$ is a maximal subset of $\Sigma^{sc}(G)$ that is adapted to $\wm$. Let $\Sigma'$ be a subset of $\Sigma^{sc}(G)$ that is adapted to $\wm$ and such that $\Sigma^{sc}(X) \inn \Sigma'$. By Definition~\ref{def:adapted}, there exits an affine variety $Y$ with weight monoid $\wm$ such that $\Sigma^{sc}(Y)=\Sigma'$. Proposition~\ref{prop:orbincl} now implies that $T\cdot[X] \inn \overline{T\cdot[Y]}$. Since $X$ is generic, which means that $\overline{T\cdot[X]}$ is an irreducible component of $\ms$, this implies that $\overline{T\cdot[X]} = \overline{T\cdot[Y]}$. Consequently $T\cdot[X] = T\cdot[Y]$, and so $X$ is $G$-equivariantly isomorphic to $Y$ by \cite[Theorem 1.12]{alexeev&brion-modaff}. In particular, $\Sigma^{sc}(X) = \Sigma^{sc}(Y) = \Sigma'$. 

Conversely, suppose that $\Sigma^{sc}(X)$ is a maximal subset of $\Sigma^{sc}(G)$ that is adapted to $\wm$. If $X$ were not generic, then Proposition~\ref{prop:denseTorb_irrcomp} would imply the  existence of an affine spherical $G$-variety $Y$ with weight monoid $\wm$ such that $\overline{T\cdot[X]} \subsetneq \overline{T\cdot[Y]}$. But then $\Sigma^{sc}(X) \subset \Sigma^{sc}(Y)$ by Proposition~\ref{prop:orbincl}, and $\Sigma^{sc}(X) \neq \Sigma^{sc}(Y)$, since otherwise we would have $T \cdot [Y] \inn \overline{T\cdot[X]}$ again by Proposition~\ref{prop:orbincl}. This contradicts the maximality of $\Sigma^{sc}(X)$, and finishes the proof.  
\end{proof}

\begin{corollary} \label{cor:msirreducible}
$\ms$ is irreducible if and only if $\Sigma^{sc}(\wm)$ is adapted to $\wm$. 
\end{corollary}
\begin{proof}
This is a formal consequence of Proposition~\ref{prop:orbincl}.
We first assume that $\Sigma^{sc}(\wm)$ is adapted to $\wm$, and denote by $X$ an affine spherical $G$-variety such that $\wm(X) = \wm$ and $\Sigma^{sc}(X) = \Sigma^{sc}(\wm)$. We claim that $\overline{T\cdot[X]} = \ms$. To prove the claim, it suffices to show that if $Y$ is an affine spherical $G$-variety with $\wm(Y) = \wm$, then $T\cdot[Y] \inn \overline{T\cdot[X]}$. By Proposition~\ref{prop:orbincl}(\ref{orbincl_item3}) we have that $\Sigma^{sc}(Y) \inn \Sigma^{sc}(\wm)=\Sigma^{sc}(X)$. By Proposition~\ref{prop:orbincl}(\ref{orbincl_item1}), it follows that $T\cdot[Y] \inn \overline{T\cdot[X]}$.   
 
We turn to the the reverse implication. Since $\ms$ is irreducible, it has a unique dense $T$-orbit $T\cdot[X]$. In particular
\begin{equation}
\overline{T\cdot[X]} = \ms. \label{eq:doims}
\end{equation}
 We claim that $\Sigma^{sc}(X) = \Sigma^{sc}(\wm)$. By Proposition~\ref{prop:orbincl}(\ref{orbincl_item3}) we only have to show that $\Sigma^{sc}(\wm)\inn \Sigma^{sc}(X)$.  Let $\sigma \in \Sigma^{sc}(\wm)$. By the definition of $\Sigma^{sc}(\wm)$ there exists an affine spherical $G$-variety $Y$ with $\Sigma^{sc}(Y) = \{\sigma\}$ and $\wm(Y) = \wm$. By the equality (\ref{eq:doims}), it follows that $T\cdot[Y] \inn \overline{T\cdot[X]}.$ Proposition~\ref{prop:orbincl}(\ref{orbincl_item1}) finishes the proof.   
\end{proof}

\begin{remark} \label{rem:sigmascadapted}
Note that $\Sigma^{sc}(\wm)$ is adapted to $\wm$ if $\wm$ is $G$-saturated, and when $\Sigma^{sc}(\wm)$ does not contain any simple roots; see Corollary~\ref{cor:allsimdoub_msirr} below.
\end{remark}

\subsection{Spherical roots adapted to a weight monoid} \label{subsec:spherroots_adapted_weightmonoid}
In this subsection, we begin by recalling some results from \cite{msfwm}, including the combinatorial characterization of $\sigma \in \Sigma^{sc}(G)$ that are adapted to $\wm$; see Proposition~\ref{prop:adapsr}. We proceed with  a proof of a criterion formulated by Luna which characterizes the subsets of $\Sigma^{sc}(G)$ that are adapted to $\wm$; see Proposition~\ref{prop:luna-adapt}.

We begin by  introducing some notation which we will use in this subsection. Let $\wm$ be a normal submonoid of $\dw$.  
We will denote the dual cone to $\wm$ by $\wm^{\vee}$, that is,
\[\wm^{\vee}:=\{v\in\Hom_{\Z}(\Z\wm,\Q):
  \langle v,\gamma\rangle\geq0\mbox{ for all }\gamma\in\wm\}.\]
It is a strictly convex polyhedral cone, and we denote the set of primitive vectors
on its rays by $E(\wm)$:
\begin{equation}
E(\wm):= \{\delta \in (\Z\wm)^* \colon \delta
\text{ spans a ray of }\wm^\vee \text{ and } \delta \text{ is primitive}\}.
\end{equation}
Observe that
\begin{multline} \label{eq:1}
E(\wm)=\{\delta \in (\Z\wm)^*\colon \delta  \text{
    is primitive},
\delta(\wm) \inn \Z_{\ge 0}, \\
\delta \text{ is the equation of a face
  of codim $1$ of $\Q_{\ge 0}\wm$}\}. \end{multline} 
For 
$\alpha \in S \cap \Z\wm$, we define
\[a(\alpha):= \{\delta \in  (\Z\wm)^*\colon \<\delta,\alpha\>=1 \text{
  and } \bigl(\delta
\in E(\wm) \text{ or } \alpha^{\vee}|_{\Z\wm} - \delta \in E(\wm)\bigr)\}.\]

The next three results are taken from \cite{msfwm}. Before we state them, we recall from \cite{luna-typeA} the definition of the colors, and of an augmentation, of a spherical system. We use the formulation of \cite{msfwm}. 

\begin{definition} \label{def:colors_sph_system}
Let $\s=(S^p,\Sigma,\mathbf A)$ be a (spherically closed) spherical $G$-system. The
\textbf{set of colors of $\s$} is the finite set $\Delta$
obtained as the disjoint union $\Delta=\Delta^a\cup\Delta^{2a}\cup\Delta^b$ where:
\begin{itemize}
\item $\Delta^a=\mathbf A$,
\item $\Delta^{2a}=\{D_\alpha : \alpha\in S\cap{\frac1 2}\Sigma\}$,
\item $\Delta^b=\{D_\alpha : \alpha\in S\setminus(S^p\cup\Sigma\cup{\frac1 2}\Sigma)\}/\sim$, where $D_\alpha\sim D_\beta$ if $\alpha$ and $\beta$ are orthogonal and $\alpha+\beta\in\Sigma$.
\end{itemize} 
\end{definition}

\begin{definition} \label{def:augmentation}
Let $\mathscr S=(S^p,\Sigma,\mathbf A)$ be a spherically closed
spherical $G$-system with Cartan pairing $c: \Z \mathbf{A} \times
\Z\Sigma \to \Z$. 
An \textbf{augmentation} of $\mathscr S$ is a
lattice $\Lambda'\subset\Lambda$ endowed with a pairing
$c'\colon\Z\mathbf A\times\Lambda'\to\Z$ such that
$\Lambda'\supset\Sigma$ and
\begin{itemize}
\item[(a1)] $c'$ extends $c$;
\item[(a2)] if $\alpha\in S\cap\Sigma$ then $c'(D_\alpha^+,\xi)+c'(D_\alpha^-,\xi)=\langle\alpha^\vee,\xi\rangle$ for all $\xi\in\Lambda'$;
\item[($\sigma1$)] if $2\alpha\in2S\cap\Sigma$ then $\alpha \notin
  \Lambda'$ and $\langle\alpha^\vee,\xi\rangle\in2\Z$ for all $\xi\in\Lambda'$;
\item[($\sigma2$)] if $\alpha$ and $\beta$ are orthogonal elements of
  $\sr$ with $\alpha+\beta\in\Sigma$ then
  $\langle\alpha^\vee,\xi\rangle=\langle\beta^\vee,\xi\rangle$ for all
  $\xi\in\Lambda'$; and
\item[(s)] if $\alpha\in S^p$ then $\langle\alpha^\vee,\xi\rangle=0$ for all $\xi\in\Lambda'$.
\end{itemize}
Let $\col$ be the set of colors of $\mathscr S$. The \textbf{full
  Cartan pairing} of the augmentation is the $\Z$-bilinear map $c': \Z \col
\times \Lambda'\to\Z$ given by
\begin{equation} \label{eq:4}
c'(D,\gamma)=\left\{\begin{array}{ll}
c'(D,\gamma) & \mbox{ if $D\in\col^a$}; \\
{\frac1  2}\langle\alpha^\vee,\gamma\rangle & \mbox{ if $D=D_\alpha\in\col^{2a}$}; \\
\langle\alpha^\vee,\gamma\rangle & \mbox{ if $D=D_\alpha\in\col^b$}. 
\end{array}\right.
\end{equation}
\end{definition}

\begin{remark} \label{rem:colors_variety_system}
Let $X$ be a spherical $G$-variety. The set of colors of $X$ is naturally identified with the set of colors of $\mathscr{S}(X)$, thanks to \cite[Proposition 3.2]{luna-typeA}. The lattice $\lat(X)$ together with the Cartan pairing $c_X$ is an augmentation of $\mathscr{S}(X)$, thanks to \cite[Proposition 6.4]{luna-typeA}. 
\end{remark}

\begin{proposition}[{\cite[Proposition 2.13 and Remark 2.14]{msfwm}}]  \label{prop:adapt-spher-roots}
Let $\wm$ be a normal submonoid
of $\dw$. Suppose that a subset $\Sigma$ of $\Sigma^{sc}(G)$ is adapted to $\wm$, let $X$ be as in Definition~\ref{def:adapted}, and set $S^p=S^p(X)$, $\mathbf A=\mathbf A(X)$. Then $\mathscr S=(S^p, \Sigma, \mathbf A)$ satisfies
\begin{enumerate}
\item  $S^p=S^p(\wm)$; and \label{item:1}
\item  $\Z\wm$ is the lattice of an augmentation of $\mathscr S$, such that \label{item:2}
\item if $\delta \in E(\wm)$, then $\<\delta, \sigma\> \leq 0$ for all $\sigma \in
\Sigma$ or there exists $D\in \col$ such that $c(D, \cdot)$ is a
positive multiple of $\delta$; where $\col$ is the set of colors of
$\mathscr S$ and $c: \Z\col \times \Z\wm \to \Z$ is the full Cartan
pairing of the augmentation; and \label{item:3}
\item $ c(D,\cdot) \in \wm^\vee$ for all $D \in \mathbf{A}$. \label{item:4}
\end{enumerate}
Viceversa, let $\Sigma$ be a subset of $\Sigma^{sc}(G)$ and suppose that there exists a spherically closed spherical system $\mathscr S=(S^p, \Sigma, \mathbf A)$ with the properties (\ref{item:1})--(\ref{item:4}). Then $\Sigma$ is adapted to $\wm$, and $\mathscr S$ and the augmentation are uniquely determined by these properties.
\end{proposition}

\begin{remark} 
\begin{enumerate}[(a)]
\item The proof of the fact that $\Sigma$ is adapted to $\wm$ in the viceversa statement of Proposition~\ref{prop:adapt-spher-roots} relies on the Luna Conjecture, that is on \cite[Theorem 1.2.3]{bravi-pezzini-primwonderful}.
\item In condition (\ref{item:4}) of Proposition~\ref{prop:adapt-spher-roots}, we could replace $\mathbf{A}$ by the set $\col$ of all colors of $\mathscr S$. Indeed, if $D \in \col \setminus \mathbf{A}$, then $c(D, \cdot)$ takes the same values on $\Z\wm$ as a coroot or its half, and therefore takes nonnegative values on $\wm \inn \dw$.
\end{enumerate} 
\end{remark}

The following lemma, extracted from the proof of \cite[Corollary 2.15]{msfwm}, explains the ``meaning'' of the set $a(\alpha)$.
\begin{lemma} \label{lem:a_alpha}
Let $\s = (S^p, \Sigma, \A)$ and $(\Z\wm,c)$ be the spherically closed spherical system and the augmentation as in Proposition~\ref{prop:adapt-spher-roots}. If $\alpha \in \Sigma \cap \sr$ and $\A(\alpha)=\{D_{\alpha}^+, D_{\alpha}^-\}$, then
\begin{equation}
a(\alpha) = \{c(D_{\alpha}^+, \cdot), c(D_{\alpha}^-,\cdot)\}. \label{eq:a_alpha_meaning}
\end{equation}
\end{lemma}

While Proposition~\ref{prop:adapt-spher-roots} depends on the Luna Conjecture, the following combinatorial characterization of $\sigma \in \Sigma^{sc}(G)$ that are adapted to $\wm$ only uses the classification of spherical varieties of rank $1$ \cite{ahiezer-eqvcompl,brion-rank1}. 
\begin{proposition}[{\cite[Corollary 2.16]{msfwm}}]  \label{prop:adapsr}
Let $\wm$ be a normal monoid of dominant weights. An element $\sigma$ of $\Sigma^{sc}(G)$ is adapted to $\wm$ if and only if all of the following conditions hold:
\begin{enumerate}[(1)]
\item $\sigma \in \Z\wm $; \label{item:inasr1}
\item $\sigma$ is compatible with $S^p(\wm)$; \label{item:inasr2}
\item if $\sigma \notin \sr$ and $\delta \in E(\wm)$ such that
  $\<\delta, \sigma\> > 0$ then there exists $\beta \in S\setminus
  S^p(\wm)$ such that $\beta^\vee$ is a positive multiple of
  $\delta$; \label{item:inasr3}
\item if $\sigma \in \sr$ then   \label{item:inasr4}
\begin{enumerate}[(a)]
\item $a(\sigma)$ has one or two elements; and    \label{item:inasr4a}
\item $\<\delta, \gamma\> \ge 0$ for all $\delta \in a(\sigma)$ and
  all $\gamma \in \wm$; and     \label{item:inasr4b}
\item $\<\delta, \sigma\> \le 1$ for all $\delta \in E(\wm)$; \label{item:inasr4c}
\end{enumerate}
\item if $\sigma = 2\alpha \in 2\sr$, then $\alpha \notin \Z\wm$ and  $\<\alpha^{\vee}, \gamma\>
\in 2\Z$ for all $\gamma \in \wm$; \label{item:inasr5}
\item if $\sigma = \alpha + \beta$ with $\alpha,\beta \in \sr$ and \label{nsorths}
  $\alpha \perp \beta$, then $\alpha^{\vee} = \beta^{\vee}$ on $\wm$.
\end{enumerate}
\end{proposition}

The following criterion was formulated by Luna in 2005 in an unpublished note.  
\begin{proposition} \label{prop:luna-adapt}
Let $\wm$ be a normal monoid of dominant weights. A subset $\Sigma$ of $\Sigma^{sc}(G)$ is adapted to $\wm$ if and only if the following two conditions hold:
\begin{enumerate}[(a)]
\item $\Sigma$ is a subset of $\Sigma^{sc}(\wm)$; \label{adapt1}
\item If $\alpha \in \sr \cap \Sigma$, $\delta \in a(\alpha)$ and $\gamma \in \Sigma$ satisfy $\<\delta,\gamma\> >0$, then $\gamma \in \sr$ and $\delta \in a(\gamma)$. \label{adapt2}
\end{enumerate}
\end{proposition}
\begin{proof}
We first prove  the necessity of the two conditions. Assume that $\Sigma$ is adapted to $\wm$. Condition (\ref{adapt1}) is Proposition~\ref{prop:orbincl}(\ref{orbincl_item3}). Thanks to Lemma~\ref{lem:a_alpha}, condition (\ref{adapt2}) follows from Luna's axiom (A1).  

To show that the two conditions are sufficient, we will use Proposition~\ref{prop:adapt-spher-roots}. We 
first construct a triple $\s = (S^p, \Sigma, \A)$ and a pairing $c\colon \Z\A \times \Z\wm \to \Z$, and then show that (\ref{adapt1}) and (\ref{adapt2}) imply that $\s = (S^p, \Sigma, \A)$ and $c$ satisfy all the conditions in Proposition~\ref{prop:adapt-spher-roots}. Note that (\ref{adapt1}) means exactly that every $\sigma \in \Sigma$ satisfies conditions (\ref{item:inasr1}) -- (\ref{nsorths}) in Proposition~\ref{prop:adapsr}.  
 
We put $S^p:=S^p(\wm)$. For every $\alpha \in \Sigma \cap \sr$ we formally put $\A(\alpha):=\{D_{\alpha}^+, D_{\alpha}^-\}$. If $a(\alpha)$ has one element, then we put $c(D_{\alpha}^+,\cdot) := c(D_{\alpha}^-,\cdot):= \frac{1}{2}\alpha^{\vee}|_{\Z\wm}$. If $a(\alpha)$ has two elements, say $a(\alpha) = \{\delta_{\alpha}^+, \delta_{\alpha}^-\}$, then we set $c(D_{\alpha}^+,\cdot):=\delta_{\alpha}^+$ and $c(D_{\alpha}^-,\cdot):=\delta_{\alpha}^-$. Finally, we put 
\[
\A:=\frac{\coprod_{\alpha \in \Sigma^{sc} \cap \sr} \A(\alpha)}{\sim}
\]
where $D_1 \sim D_2$ if there exist $\alpha,\beta \in \Sigma \cap \sr$ such that $\alpha\neq\beta$, $D_1 \in \A(\alpha)$, $D_2 \in \A(\beta)$ and $c(D_1,\cdot) = c(D_2,\cdot)$. 

\underline{Step 1:} We check that $\s = (S^p, \Sigma, \A)$ is a spherically closed spherical system. Stricly speaking, this triple is equipped with the restriction of $c:\Z\A \times \Z\wm \to \Z$ to $\Z\A \times \Z\Sigma$, but we will also denote this restriction by $c$, since no confusion will arise.  We begin by verifying axiom (A1). Let $D \in \A(\alpha)$ for some $\alpha \in \Sigma \cap \sr$ and let $\gamma \in \Sigma$. Then $c(D,\cdot) \in a(\alpha)$.  If $c(D,\gamma)>0$, then $\gamma \in \sr$ and $c(D,\cdot) \in a(\gamma)$ by (\ref{adapt2}). By the definition of $a(\gamma)$, it follows that $c(D,\gamma)=1$ and so (A1) holds.

Axioms (A2) and (A3) hold by the construction of $\A$, where we idenitfy $\A(\alpha)$ with its image in $\A$. Axiom ($\Sigma$2) follows from (\ref{nsorths}) in Proposition~\ref{prop:adapsr}. Axiom (S) follows from (\ref{item:inasr2}) in Proposition~\ref{prop:adapsr}. 

Next, we turn to axiom ($\Sigma$1). Let $2\alpha \in \Sigma \cap 2\sr$ and let $\sigma \in \Sigma \setminus \{2\alpha\}$. The fact that $\<\frac{1}{2}\alpha^{\vee}, \sigma\> \in \Z$ follows  from (\ref{item:inasr5}) in Proposition~\ref{prop:adapsr}. We need to show that 
\begin{equation}
\<\alpha^{\vee}, \sigma\> \leq 0,   
\end{equation}
but this follows from Lemma~\ref{lem:coroot_nonpos} below.

\underline{Step 2:} We now check that $(\Z\wm, c)$ is an augmentation of $\s$; that is, we check all the conditions of Definition~\ref{def:augmentation}.  Since $\Sigma \inn \Sigma^{sc}(\wm)$, it follows from Proposition~\ref{prop:adapsr}(\ref{item:inasr1}) that $\Z\Sigma \inn \Z\wm$. Axiom (a1) holds because $\s$ was equipped with the restriction of $c:\Z\A \times \Z\wm \to \Z$ to $\Z\A \times \Z\Sigma$. Axiom (a2) holds by the construction of $c$. Axiom ($\sigma 1$) is an immediate consequence of Proposition~\ref{prop:adapsr}(\ref{item:inasr5}). Similary, axiom ($\sigma 2$) follows from Proposition~\ref{prop:adapsr}(\ref{nsorths}). Axiom (s), finally, is true by the definition of $S^p = S^p(\wm)$, cf.\ Definition~\ref{def:Sp}.

\underline{Step 3:} We verify condition (\ref{item:3}) of Proposition~\ref{prop:adapt-spher-roots}. Let $\delta \in E(\wm)$ and $\sigma \in \Sigma$ such that $\<\delta, \sigma\> > 0$. Let $\col$ be the set of colors of $\s$. We have to prove that there exists $D \in \col$ such that $c(D,\cdot)$ is a positive multiple of $\delta$. We will consider two cases:
\begin{enumerate}[(i)]
\item $\sigma \notin \sr$; \label{item:condition3c1}
\item $\sigma \in \sr$. \label{item:condition3c2}
\end{enumerate}

Suppose we are in case (\ref{item:condition3c1}). Then Proposition~\ref{prop:adapsr}(\ref{item:inasr3}) tells us there exists $\beta \in \sr \setminus S^p$ such that $\beta^{\vee}|_{\Z\wm} \in \Q_{>0} \delta$. If $\beta \notin \Sigma$, then the construction of the full Cartan pairing of $\s$ implies that there exists $D \in \col$ such that $c(D,\cdot)$ is equal to $\beta^{\vee}|_{\Z\wm}$ or to $\frac{1}{2}\beta^{\vee}|_{\Z\wm}$. It follows that $c(D,\cdot)$ is a positive rational multiple of $\delta$. On the other hand we claim that $\beta \in \Sigma$ is impossible. Indeed, if $\beta$ were an element of $\Sigma$, then $\<\delta, \beta\>=1$ by (\ref{item:inasr4c}) of Proposition~\ref{prop:adapsr} and consequently $\delta \in a(\beta)$, which would imply, by (\ref{adapt2}) of the present proposition, that $\sigma \in \sr$. But this contradicts our assumption (\ref{item:condition3c1}). 

We now consider case (\ref{item:condition3c2}). Then $\delta \in a(\sigma)$ by Proposition~\ref{prop:adapsr} (\ref{item:inasr4c}). By the construction of $c$ above in this proof, it follows that $\delta = c(D,\cdot)$ for at least one of the colors $D \in \A(\alpha)$.

\underline{Step 4:} Finally, we verify condition (\ref{item:4}) of Proposition~\ref{prop:adapt-spher-roots}. Suppose $D \in \A(\alpha)$ for some $\alpha \in \sr \cap \Sigma$. If $|a(\alpha)|=1$, then $c(D,\cdot)$ is a positive rational multiple of a coroot and therefore takes nonnegative values on $\wm \inn \dw$. If $|a(\alpha)|=2$, then we conclude, through (\ref{adapt1}), by condition (\ref{item:inasr4b}) of Proposition~\ref{prop:adapsr}.  
\end{proof}

\begin{lemma} \label{lem:coroot_nonpos}
Let $\wm$ be a normal submonoid of $\dw$ and let $\sigma, \beta \in
\Sigma^{sc}(\wm)$ with $\sigma \neq \beta$. If $\beta = 2\alpha \in
2\sr$, or $\beta = \alpha \in \sr$ with $|a(\alpha)|=1$, then
\begin{equation}
\<\alpha^{\vee}, \sigma\> \leq 0. \label{eq:alphaveesigleq0}  
\end{equation}
\end{lemma}
\begin{proof}
We prove the lemma case-by-case for the different types of spherically closed spherical roots $\sigma$. For all but two of the types (the sperical root $\alpha_1+\alpha_2$ with support of type $\sB_2$ and $\alpha_1 + 2(\alpha_2 + \ldots + \alpha_{n-1})+\alpha_n$ with $\supp(\sigma)$ of type $C_n$), we do this by showing that 
\begin{equation}
\alpha \notin \supp(\sigma). \label{eq:alphanotinsupp}
\end{equation}
This implies \eqref{eq:alphaveesigleq0} because $\sigma$ is a linear combination with positive coefficients of the simple roots that make up $\supp(\sigma)$. 

We will use a few times that the hypotheses on $\beta$ imply that 
\begin{equation}
\<\alpha^{\vee}, \delta\> \in 2\Z \text{ for all $\delta \in \Z\wm$} \label{eq:alphavdeltaeven}. 
\end{equation}
If $\beta = 2\alpha$, then (\ref{eq:alphavdeltaeven}) is part of Proposition~\ref{prop:adapsr}(\ref{item:inasr5}). On the other hand, if $\beta = \alpha$ with $|a(\alpha)|=1$, then it follows form the definition of $a(\alpha)$ that $a(\alpha) = \{\frac{1}{2} \alpha^{\vee}|_{\Z\wm}\}$ and that $\frac{1}{2} \alpha^{\vee}|_{\Z\wm}$ takes integer values on $\Z\wm$. This means that (\ref{eq:alphavdeltaeven}) holds in this case as well. 

For some of the types of spherical roots, we will take advantage of the following consequence of the fact that $\beta \in \Z\wm$:
\begin{equation}
\<\alpha^{\vee}, \gamma\> = 0 \text{ for all $\gamma \in S^p(\wm)$}. \label{eq:spperp}
\end{equation}

We now proceed with the case-by-case verification. 
\begin{itemize}
\item[-] $\sigma \in \sr$: it is enough to show that $\sigma \neq \alpha$, since then \eqref{eq:alphanotinsupp} is trivial. If $\beta=\alpha$ then $
\sigma\neq \alpha$ holds by assumption. If $\beta = 2\alpha$, then $2\alpha \in \Sigma^{sc}(\wm)$ and therefore, by Proposition~\ref{prop:adapsr}(\ref{item:inasr5}), $\alpha \notin \Z\wm$, and in particular $\alpha \notin \Sigma^{sc}(\wm) \ni \sigma$. 

\item[-] $\sigma \in 2\sr$: it is enough to show that $\sigma \neq 2\alpha$. If $\beta = 2\alpha$ then this is true by assumption. If $\beta = \alpha$ then it follows from Proposition~\ref{prop:adapsr}(\ref{item:inasr1}) that $\alpha \in \Z\wm$ and then from (\ref{item:inasr5}) in the same Proposition that $2\alpha \notin \Sigma^{sc}(\wm)$.

\item[-] $\sigma = \alpha' + \beta'$ with $\supp(\sigma)$ of type $\sA_1 \times \sA_1$: since $\beta \in \Sigma^{sc}(\wm)$, Proposition~\ref{prop:adapsr}(\ref{item:inasr1}) implies that $\beta \in \Z\wm$. Because $\sigma \in \Sigma^{sc}(\wm)$ it then follows from Proposition~\ref{prop:adapsr}(\ref{nsorths}) that $\alpha \notin \supp(\sigma)=\{\alpha',\beta'\}$.  

\item[-] $\sigma =\alpha_1+\ldots+\alpha_n$ with $\supp(\sigma)$ of type $\sA_n$, $n\geq2$: since $\sigma$ is compatible with $S^p(\wm)$, the subset $\{\alpha_2, \alpha_3, \ldots, \alpha_{n-1}\}$ of $\supp(\sigma)$ belongs to $S^p(\wm)$. If $n\ge 3$, this implies, using~\eqref{eq:spperp}, that $\alpha \notin \supp(\sigma)$. We now consider the case $n=2$. Then $\supp(\sigma) = \{\alpha_1, \alpha_2\}$ and \(
\<\alpha_1^{\vee},\sigma\> =1 = \<\alpha_2^{\vee},\sigma\>. 
\)
Since $\sigma \in \Z\wm$, it follows from (\ref{eq:alphavdeltaeven}) that $\alpha \notin \supp(\sigma)$. 

\item[-] $\sigma = \alpha_1+\ldots+\alpha_n$ with $\supp(\sigma)$ of type $\mathsf B_n$, $n\geq 2$: for $n\ge 3$, the argument that $\alpha \notin \supp(\sigma)$ is the same as for the previous spherical root. When $n=2$, then $\alpha\not=\alpha_1 \in \supp(\sigma)$ by (\ref{eq:alphavdeltaeven}), since $\<\alpha_1^{\vee},\sigma\> =1$. If $\alpha = \alpha_2 \in \supp(\sigma)$, then $\<\alpha^{\vee}, \sigma\> = 0$ and so~\eqref{eq:alphaveesigleq0} holds.   
\item[-] $\sigma = \alpha_1 + \alpha_2$ with $\supp(\sigma)$ of type $\sG_2$: in this case $\alpha \notin \supp(\sigma) = \{\alpha_1, \alpha_2\}$ by (\ref{eq:alphavdeltaeven}) since $\<\alpha_1^{\vee}, \sigma\> = -1 \notin 2\Z$ and $\<\alpha_2^{\vee}, \sigma\>=1 \notin 2\Z$. 

\item[-] $\sigma = \alpha_1 + 2(\alpha_2 + \ldots + \alpha_{n-1})+\alpha_n$ with $\supp(\sigma)$ of type $C_n$, $n\geq 3$: it follows from the compatibility of $\sigma$ with $S^p(\wm)$ that $\{\alpha_3, \alpha_4, \ldots, \alpha_n\} \inn S^p(\wm) \cap \supp(\sigma)$. This implies, using (\ref{eq:spperp}), that either $\alpha \notin \supp(\sigma)$ or $\alpha  = \alpha_1 \in \supp(\sigma)$. Since $\<\alpha_1^{\vee}, \sigma\>=0$, equation \eqref{eq:alphaveesigleq0} holds either way. 

\item[-] The remaining six types of spherically closed spherical roots are all handled in the same way: if $\sigma$ is of one of these types, then it follows from 
the compatibility of $\sigma$ with $S^p(\wm)$, for each type, that $S^p(\wm)$ contains all but one of the simple roots in $\supp(\sigma)$. It then easily follows that (\ref{eq:spperp}) cannot hold for any $\alpha \in \supp(\sigma)$. To be more precise:
\begin{itemize}
\item[$\cdot$] if $\sigma = \alpha_1 + 2\alpha_2 + \alpha_3$ with $\supp(\sigma)$ of type $\sA_3$, then $S^p(\wm) \cap \supp(\sigma)$ contains $\alpha_1$ and $\alpha_3$;  
\item[$\cdot$] if $\sigma = 2(\alpha_1 +\ldots + \alpha_n)$ with $\supp(\sigma)$ of type $\sB_n$ where $n\geq 2$, then $S^p(\wm) \cap \supp(\sigma)$ contains $\{\alpha_2, \alpha_3, \ldots, \alpha_n\}$;
\item[$\cdot$] if $\sigma = \alpha_1 + 2\alpha_2 + 3\alpha_3$ with $\supp(\sigma)$ of type $\sB_3$, then $S^p(\wm) \cap \supp(\sigma)$ contains $\alpha_1$ and $\alpha_2$;
\item[$\cdot$] if $\sigma = 2(\alpha_1 + \ldots + \alpha_{n-2})+\alpha_{n-1} + \alpha_n$ with $\supp(\sigma)$ of type $\sD_n$ where $n\geq 4$, then $S^p(\wm) \cap \supp(\sigma)$ contains $\{\alpha_2, \alpha_3, \ldots, \alpha_n\}$;
\item[$\cdot$] if $\sigma = \alpha_1 + 2\alpha_2 + 3\alpha_3+2\alpha_4$ with $\supp(\sigma)$ of type $\sF_4$, then $S^p(\wm) \cap \supp(\sigma)$ contains $\{\alpha_1,\alpha_2, \alpha_3\}$;
\item[$\cdot$] if $\sigma = 4\alpha_1+2\alpha_2$ with $\supp(\sigma)$ of type $\sG_2$, then $S^p(\wm) \cap \supp(\sigma)$ contains $\alpha_2$.
\end{itemize}
\end{itemize}
\end{proof}

\subsection{$G$-saturated weight monoids} \label{subsec:Gsat}
In this section, we will look at the combinatorics of affine spherical varieties with $G$-saturated weight monoids.

\begin{proposition} \label{prop:Gsatnosimadap}
If $\wm$ is a $G$-saturated submonoid of $\dw$, then $|a(\sigma)|=1$ for every $\sigma \in \Sigma^{sc}(\wm) \cap \sr$.
\end{proposition}
\begin{proof}
We will prove the contrapositive. Recall from Proposition~\ref{prop:adapsr} that $a(\sigma)$ has one or two elements for every $\sigma \in \Sigma^{sc}(\wm) \cap \sr$.   Let $\sigma\in \Sigma^{sc}(\wm)\cap S$ with $|a(\sigma)|=2$. We will show that there then exists a dominant weight $\gamma'$ in $\Z \wm$ with $\gamma' \not\in \wm$. 

Let $\delta\in E(\wm)$ such that $a(\sigma)=\{\delta, \sigma^\vee|_{\Z\wm} - \delta\}$. It follows from $|a(\sigma)|=2$ that $\delta$ and $\sigma^{\vee}|_{\Z\wm}$ are not proportional: if they were, then $\<\delta,\sigma\> = 1 = \<\sigma^{\vee}|_{\Z\wm}-\delta, \sigma\>$ would imply $|a(\sigma)|=1$. 

Since $\delta \in E(\wm)$, the set $\{\gamma \in \wm \colon \<\delta,\gamma\> = 0\}$ spans $\ker \delta$, which is a sublattice of corank $1$ of $\Z\wm$. As $\delta$ and $\sigma^{\vee}|_{\Z\wm}$ are not proportional and therefore have different kernels, this implies that there exists $\gamma \in \wm$ such that $\<\delta, \gamma\> = 0$ and $\<\sigma^{\vee}, \gamma\> \neq 0$. Then $\<\sigma^{\vee}, \gamma\> \geq 0$ since $\gamma$ is a dominant weight. 

Let $\gamma':= 2\gamma - \sigma$. Clearly $\gamma' \in \Z\wm$. Moreover, $\<\alpha^{\vee}, \gamma'\> \ge 0$ for all $\alpha \in S\setminus\{\sigma\}$ since $\gamma$ is dominant and $\<\alpha^{\vee},\sigma\> \le 0$. Furthermore $\<\sigma^{\vee}, \gamma'\> = \<\sigma^{\vee}, 2\gamma\> - \<\sigma^{\vee}, \sigma\> \ge 2-2 = 0$. Consequently, $\gamma'$ is a dominant weight. 

On the other hand $\<\delta, \gamma'\> = 2\<\delta,\gamma\> - \<\delta,\sigma\> = -1$ which implies that $\gamma' \notin \wm$. This proves that $\wm$ is not $G$-saturated. 
\end{proof}

\begin{corollary} \label{cor:allsimdoub_msirr}
If $|a(\alpha)|=1$ for all $\alpha \in \Sigma^{sc}(\wm) \cap \sr$, then 
 $\Sigma^{sc}(\wm)$ is adapted to $\wm$. In particular, $\Sigma^{sc}(\wm)$ is adapted to $\wm$ if $\wm$ is $G$-saturated. 
\end{corollary}
\begin{proof}
The first assertion follows from  Proposition~\ref{prop:luna-adapt} and Lemma~\ref{lem:coroot_nonpos}; indeed, the Lemma implies that condition (\ref{adapt2}) of the Proposition is trivially met. The second assertion follows from the first, by Proposition~\ref{prop:Gsatnosimadap}. 
\end{proof}

\begin{corollary} \label{cor:ms_sat_irr}
If $\wm$ is a $G$-saturated submonoid of $\dw$, then 
\begin{enumerate}[(a)]
\item $\ms$ is irreducible; \label{cor:ms_sat_irr_a}
\item up to $G$-equivariant isomorphism, there is exactly one \emph{generic} affine spherical $G$-variety $X_{\wm}$ with weight monoid $\wm$; \label{cor:ms_sat_irr_b}
\item $\Sigma^{sc}(X_{\wm}) = \Sigma^{sc}(\wm)$. \label{cor:ms_sat_irr_c}
\end{enumerate}
\end{corollary}
\begin{proof}
Assertion (\ref{cor:ms_sat_irr_a}) follows from Corollaries \ref{cor:msirreducible} and \ref{cor:allsimdoub_msirr}. Assertion (\ref{cor:ms_sat_irr_b}) follows from (\ref{cor:ms_sat_irr_a}) and Proposition~\ref{prop:denseTorb_irrcomp}. Assertion (\ref{cor:ms_sat_irr_c}), finally, is a consequence of Proposition~\ref{prop:orbincl}. 
\end{proof}

We now  give the proof of the Proposition~\ref{prop:adapnsphroots} on page~\pageref{prop:adapnsphroots}. Before doing so, we recall the following Definition from \cite{msfwm}. 

\begin{definition}\label{def:Nadapted}
Let $\wm$ be a normal submonoid of $\dw$. We say that an element $\sigma$ of $\Sigma^{sc}(G)$ is  \textbf{N-adapted to $\wm$} 
if there exists an affine spherical variety $X$ such that $\wm(X)=\wm$ and $\Sigma^{N}(X)=\{\sigma\}$. We use $\Sigma^{N}(\wm)$ for the set of all $\sigma \in \Sigma^{sc}(G)$ that are N-adapted to $\wm$.  
\end{definition}

\begin{proof}[Proof of Proposition~\ref{prop:adapnsphroots}] \label{proofpropadapn}
This is a special case of \cite[Corollary 2.17]{msfwm}. Since $\wm$ is $G$-saturated, condition (4) of that Corollary is equivalent to (\ref{item:osignotsimple}) of Proposition~\ref{prop:adapnsphroots}, by Proposition~\ref{prop:Gsatnosimadap}. Condition (3) of the Corollary is redundant, by the Definition~\ref{def:G-saturated} of a $G$-saturated weight monoid. Conditions (1), (2), (5) and (6) of the Corollary are identical to conditions (\ref{item:osiginlat}), (\ref{compat}), (\ref{item:corooteven}) and (\ref{item:osigorthosum}), respectively, of  Proposition~\ref{prop:adapnsphroots}.      
\end{proof}

\begin{proposition} \label{prop:Nspherroots_Xwm}
If $\wm$ is a $G$-saturated submonoid of $\dw$ and $X_{\wm}$ is as in Corollary~\ref{cor:ms_sat_irr}(\ref{cor:ms_sat_irr_b}), then $\Sigma^N(X_{\wm}) = \Sigma^N(\wm)$. 
\end{proposition}
\begin{proof}
The proposition follows from Corollary~\ref{cor:ms_sat_irr}(\ref{cor:ms_sat_irr_c}) and Proposition~\ref{prop:doubling_sr} by comparing Propositions \ref{prop:adapsr} and \ref{prop:adapnsphroots}. 
\end{proof}

We end this section with a proposition formulated by Luna in his aforementioned 2005 working document. Together with equation~\eqref{eq:a_alpha_meaning} it gives a geometric characterization of the affine spherical $G$-varieties with $G$-saturated weight monoid. 
\begin{proposition}[Luna] \label{prop:lunaXGsat}
Let $X$ be an affine spherical $G$-variety with open $G$-orbit $X_o$. The weight monoid $\wm(X)$ of $X$ is $G$-saturated if and only if the following two conditions hold:
\begin{enumerate}[(1)]
\item $\mathrm{codim}_X (X\setminus X_o) \ge 2$;\label{eq:gsatlunacodim}
\item $|a(\sigma)|=1$ for every $\sigma \in \Sigma^{sc}(X) \cap \sr$. \label{eq:gsatluna}
\end{enumerate}
\end{proposition} 
\begin{proof}
We begin with the ``only if'' part, so assume that $\wm(X)$ is $G$-saturated. Part (\ref{eq:gsatluna}) follows from Proposition~\ref{prop:Gsatnosimadap}, because $\Sigma^{sc}(X)\inn\Sigma^{sc}(\wm(X))$ by Proposition~\ref{prop:orbincl}. Let us show part (\ref{eq:gsatlunacodim}). The weight monoid $\wm(X_o)$ of the quasi-affine spherical $G$-variety $X_o$ is defined just like for an affine spherical variety. Since $X_o$ is a dense $G$-stable subset of $X$, we have the inclusions
\[
\wm(X)\inn\wm(X_o)\inn \Z\wm(X)\cap \dw,
\]
which together with the fact that $\wm(X)$ is $G$-saturated yield $\wm(X)=\wm(X_o)$. Since the coordinate ring $\C[X_o]$ is a multiplicity free $G$-module, this implies that $\C[X]=\C[X_o]$. If now $D$ is an irreducible subvariety of $X\setminus X_o$ of codimension $1$, then there exist two regular functions $f,g\in\C[X]$ such that $f$ vanishes on $D$, and $g$ vanishes on all irreducible components of $\{f=0\}$ except $D$. It follows that we can find a positive integer $n$ such that the quotient $g^n/f$ has no poles on $X$ except for $D$. Since $X_o$ is normal, because it's open in the normal variety $X$, this implies that $g^n/f\in\C[X_o]\setminus \C[X]$: contradiction. This proves part (\ref{eq:gsatlunacodim}). 

We prove the ``if'' part. It follows from the normality of $X$ that the monoid $\wm(X)$ is the subset of $\Z\wm(X)$ where the valuations of all colors and all $G$-stable prime divisors of $X$ take non-negative values. By condition (\ref{eq:gsatlunacodim}) we have that $X$ has no $G$-stable prime divisors. By condition (\ref{eq:gsatluna}) we know that the valuations of all colors of $X$ moved by simple roots in $\Sigma^{sc}(X)$ are actually multiples of some simple coroot of $G$. This implies that all colors of $X$ are multiples of simple coroots. Consequently $\wm(X)$ is the set of all elements of $\Z\wm(X)$ on which some simple coroots take non-negative values. This shows that $\Z\wm(X) \cap \dw \inn \wm(X)$, which implies that $\wm(X)$ is $G$-saturated. 
\end{proof}

\begin{remark}
By Proposition~\ref{prop:doubling_sr}, condition~(\ref{eq:gsatluna}) in Proposition~\ref{prop:lunaXGsat} is equivalent to 
\[\Sigma^N(X) \cap \sr = \emptyset \]
\end{remark}

\section{Camus's smoothness criterion} \label{sec:camus}
We report in this section a smoothness criterion for spherical varieties due to R.\ Camus~\cite{camus}, with a complete exposition of its original proof. 

In this section, if $A$ is any algebraic group, we denote by $\Chi(A)$ the group of its characters and we will use $A^r$ for its radical. Recall that when $G$ is a connected reductive group, $G^r$ is the connected component of the center of $G$ containing the identity. 

\begin{definition} \label{def:simple_quasivec_socle}
Let $X$ be a spherical $G$-variety.
\begin{enumerate}
\item If $X$ has a unique closed $G$-orbit then we say that $X$ is \textbf{simple}. A simple spherical variety is \textbf{quasi-vectorial} if all its colors contain the closed $G$-orbit.
\item If $X$ is simple, we denote by $\V_X$ the set of $G$-stable prime divisors of $X$, by $\col_X$ the set of the colors containing the closed $G$-orbit, and we set $\D_X=\col_X\cup \V_X$.
\item If $X$ is a simple spherical variety, its \textbf{socle} is
\[
\soc(X):=(S^p(X), \Sigma^{sc}(X), \A(X), \col_X, \V_X, \rho'_X\colon\D_X\to (\Z\Sigma^{sc}(X))^*)
\]
where $\rho'_X(D):=\rho_X(D)|_{\Z\Sigma^{sc}(X)}$ for all $D\in\D_X$.
\item Equality of socles of two different varieties is defined as equality on the components $S^p$ and $\Sigma^{sc}$, and bijections on the components $\A$, $\col$, $\V$ compatibly with the maps $\rho'$.
\end{enumerate}
\end{definition}

Notice that an affine spherical variety $X$ is simple, since $\CC[X]^G=\CC$ and therefore $X\quot G=\Spec\CC[X]^G$ is a single point.

Let $G_i$ be a connected reductive group and $X_i$ a simple spherical $G_i$-variety for all $i\in \{1,2\}$. Then the socle of the $G_1\times G_2$-variety $X_1\times X_2$ is $\soc(X_1)\times\soc(X_2)$, where the \textbf{direct product of two socles} is defined as the union of the two factors on the first five components, and defined accordingly on the component $\rho'$. 

The socle of a point under the action of a group $G$ is not considered ``trivial''; in particular, if $G$ is not abelian then the set $S^p$ is not empty. For the same reason, if a group $G=G_1\times G_2$ acts on a simple spherical variety $X$ such that the action of $G_1$ is trivial, then the socle of $X$ with respect to the $G$-action in general is not equal to the socle of $X$ considered only as a $G_2$-variety, although the difference is only that in the first case all simple roots of $G_1$ are in $S^p(X)$.

\textbf{Spherical modules}, i.e.\ finite dimensional $G$-modules that are spherical as $G$-varieties, will play a crucial role in what follows. The following well known proposition justifies the terminology "quasi-vectorial."

\begin{proposition} \label{prop:colors_E}
Let $E$ be a spherical $G$-module. Then every color of $E$ contains $\{0\}$, which is the unique closed orbit of $E$.
\end{proposition}
\begin{proof}
Clearly $\{0\}$ is a closed orbit, and since $E$ is affine and spherical, it is the only one. Let $D$ be a color of $E$. Since the $G$-action on $\CC[E]$ respects the grading and $\C[E]$ is multiplicity free, every irreducible submodule of $\CC[E]$ is homogeneous. Since $\C[E]$ is a UFD, the prime divisor $D$ has an irreducible global equation $f_D$, unique up to multiplication by a nonzero scalar. Since $D$ is $B$-stable, $f_D$ is a highest weight vector. 
It follows that $f_D$ is homogeneous, and consequently $f_D(0) = 0$. 
\end{proof}

\begin{definition}\label{def:module}
If $E$ is a spherical $G$-module, and
\[
E = \bigoplus_{i\in I} E_i
\]
a decomposition into irreducibles, then $\gamma^E_i$ (for any $i\in I$) denotes the $B$-eigenvalue of a $B$-eigenvector $f_i\in\CC[E_i]$ of degree $1$. It is also the highest weight of $E_i^*$. The set
\[
D_i := \{v \in E_i \colon f_i(v)=0\} \times \bigoplus_{j\in I\setminus\{i\}} E_j
\]
is a $B$-stable hyperplane of $E$. Write
\[
\D^1_E = \{ D_i \;|\; i\in I\},
\]
and for each $D = D_i$ as above denote $\gamma^E_i$ also by $\gamma^E_D$.
\end{definition}

\begin{lemma} \label{lemma:Sigma_N_D_1}
In the notations of Definition~\ref{def:module}, the group $\Lambda(E)$ is generated by $\Sigma^N(E)\cup\{ \gamma^E_i \;|\; i\in I\}$.
\end{lemma}
\begin{proof}
Since $\Sigma^N(E)\cup\{ \gamma^E_i \;|\; i\in I\}$ is a subset of $\lat(E)$ and since $\lat(E) = \Z\wm(E)$ it is enough to show that the subgroup $\Xi$ of $\lat(E)$ generated by $\Sigma^N(E)$ and $\{ \gamma^E_i \;|\; i\in I\}$ contains all the highest weights of $\CC[E]$. Let $M$ be a simple submodule of $\CC[E]$; then $M$ is contained in a product
\[
M \inn \prod_{i\in I} (E_i^*)^{a_i}
\]
where each $a_i\in\ZZ_{\geq0}$. We proceed by induction on $s= \sum_i a_i$. If $s=1$ then $M=E_i$ for some $i$, whence the highest weight of $M$ is $\gamma^E_i$.

We now show the induction step. Write
\begin{equation}
M \inn N E_i^* \label{eq:M_in_prod}
\end{equation}
for some $i\in I$, where
\[
N = \left(\prod_{j\in I\setminus\{i\}} (E_j^*)^{a_j}\right) \cdot (E_i^*)^{a_i-1}.
\]
By the induction hypothesis, the highest weights of all simple submodules of $N$ are in the group $\Xi$. We want to replace $N$ by an irreducible submodule in \eqref{eq:M_in_prod}, so we write a decomposition
\[
N = \bigoplus_{k} N_k
\]
of $N$ into simple submodules. The product $N E_i^*$ is the sum of the products $N_k E_i^*$, hence $M$ (which is an isotypic component of $\C[E]$, by the sphericity $E$) is contained in $N_k E_i^*$ for some $k$. By the definition of $\Sigma^N(E)$ and because both $N_k$ and $E_i^*$ are irreducible, the highest weight of $M$ is then of the form $\mu=\eta + \gamma^E_i - \sigma$, where $\eta$ is the highest weight of $N_k$ and $\sigma$ is in the monoid generated by $\Sigma^N(E)$. It follows that $\mu$ is in $\Xi$.
\end{proof}

\begin{remark} \label{rem:sigma_N_sc_E}
If $E$ is a spherical $G$-module, then $\Sigma^N(E) = \Sigma^{sc}(E)$ by Proposition~\ref{prop:doubling_sr}. Indeed, since $\C[E]$ is a UFD, no two colors can have the same valuation. 
\end{remark}

\begin{lemma}\label{lemma:socleindependent}
For every $i\in\{1,2\}$ let $G_i$ be a connected reductive group, suppose that $(G_1,G_1)\cong(G_2,G_2)$ and identify these two groups via a fixed isomorphism; choose moreover Borel subgroups $B_i\subseteq G_i$ in such a way that $B_1\cap (G_1,G_1) = B_2 \cap (G_2,G_2)$. Let $X$ be a simple spherical $G_i$-variety for every $i\in\{1,2\}$, suppose that the actions of $(G_1,G_1)$ and of $(G_2,G_2)$ coincide, and that the actions of $G_1^r$ and $G_2^r$ commute. Then the socles of $X$ with respect to the actions of $G_1$ and $G_2$ (resp.\ computed with respect to the Borel subgroups $B_1$ and $B_2$) coincide.
\end{lemma}
\begin{proof}
We may suppose that $G_i=(G_i,G_i)\times G_i^r$, whence $B_i=(B_i\cap(G_i,G_i))\times G_i^r$, and we denote $X$ by $X_i$ when the $G_i$-action is considered. Notice that the images in $\Aut(X)$ of $G_1$ and of $G_2$ normalize each other, so $G_i$ permutes the orbits of $G_{3-i}$ for all $i\in\{1,2\}$. Since these are finitely many, and $G_i$ is connected, we deduce that $G_i$ stabilizes each $G_{3-i}$-orbit. It follows that $G_1$ and $G_2$ have the same orbits on $X$, and a similar argument yields that the $B_1$-orbits and the $B_2$-orbits also coincide. As a consequence, $X$ has the same colors and the same $G_i$-invariant divisors with respect to both actions. The equalities $\A(X_1)=\A(X_2)$, $\col_{X_1}=\col_{X_2}$ and $\V_{X_1} = \V_{X_2}$ follow.

Notice that for all $i$ the parabolic subgroup $P_{X_i}$ is equal to $(P_{X_i}\cap (G_i,G_i))\times G_i^r$, and $P_{X_i}\cap (G_i,G_i)$ is the stabilizer of the open $B_i$-orbit of $X_i$ in $(G_i,G_i)$. Since the open $B_1$-orbit coincides with the open $B_2$-orbit, we have $P_{X_1}\cap (G_1,G_1)=P_{X_2}\cap (G_2,G_2)$, which yields $S^p(X_1)=S^p(X_2)$. 

Moreover, the group $B_i$ stabilizes the set $\CC(X_{3-i})^{(B_{3-i})}_{\lambda}$ for all $\lambda\in\Lambda(X_{3-i})$ and all $i\in\{1,2\}$, where $\CC(X_{3-i})^{(B_{3-i})}_{\lambda}$ denotes the set of $B_{3-i}$-eigenvectors of $B_{3-i}$-weight $\lambda$. Therefore a rational function on $X$ is a $B_1$-eigenvector if and only if it is a $B_2$-eigenvector, and if $\lambda_i$ is its $B_i$-eigenvalue, then $\lambda_1|_{B_1\cap (G_1,G_1)}=\lambda_2|_{B_2\cap (G_2,G_2)}$. Let $\Chi'$ denote the subgroup of $\Chi(B_i)$ of those elements whose restriction to $G_i^r$ is zero (notice that $\Chi'$ is naturally a sublattice of both $\Chi(B_1)$ and $\Chi(B_2)$). The  considerations above imply that $\Lambda(X_1)\cap \Chi' = \Lambda(X_2)\cap \Chi'$.

Suppose now that $X$ is quasi-affine, and consider the root monoids $\mathscr M_{X_i}$. The primitive elements in $\Lambda(X_i)$ on the extremal rays of $\QQ_{\geq 0}\mathscr M_{X_i}$ are the spherical roots of $X_i$. Now, the irreducible submodules of $\CC[X_1]$ and $\CC[X_2]$ are the same subspaces, and the two highest weights of the same irreducible submodule coincide on $B_i\cap (G_i,G_i)$. This implies that $\mathscr M_{X_1}=\mathscr M_{X_2}$, so $\Sigma(X_1)=\Sigma(X_2)$ since $\Sigma(X_i)\inn \Lambda(X_i)\cap \Chi'$.  This implies that $\Sigma^{sc}(X_1)=\Sigma^{sc}(X_2)$ by Proposition~\ref{prop:doubling_sr}. Even more is true: for all $\sigma\in\Sigma^{sc}(X_i)$ we have $\CC(X)^{(B_1)}_\sigma = \CC(X)^{(B_2)}_\sigma$, since $\sigma \in \lat(X_i)\cap \Chi'$. 

It follows that $\rho'_{X_1} = \rho'_{X_2}$, since these two maps are computed considering the vanishing of the same rational functions on $X$ (precisely, those in $\CC(X)^{(B_i)}_\sigma$ for $\sigma\in\Sigma^{sc}(X_i)$) along the same prime divisors of $X$. We have shown that $\soc(X_1)=\soc(X_2)$ when $X$ is quasi-affine. 

In general $X$ may not be quasi-affine, but being simple it is quasi-projective by a theorem of Sumihiro \cite[Theorem 1]{sumihiro-eqvcompl}. Therefore we can consider $X$ under the action of $G_1\times G_2^r$ (which naturally contains both groups $G_1$ and $G_2$ thanks to our assumptions), and we can choose an embedding of $X$ as a locally closed subset of $\PP(V)$, where $V$ is a  $G_1\times G_2^r$-module, in such a way that the embedding is both $G_1$- and $G_2$-equivariant. Let $X'\subset V$ be the cone over $X\subseteq \PP(V)$, and define $Y=X'\setminus \{0\}$.

Then $Y$ is a quasi-affine spherical $G_i\times \GG_m$-variety, and it satisfies the hypotheses of the lemma with respect to the $G_i\times \GG_m$-actions (where all invariants are computed with the Borel subgroups $B_i\times \GGm$). Thanks to the first part of the proof we have $\soc(Y_1)=\soc(Y_2)$, where $Y_i$ is $Y$ under the action of $G_i\times \GGm$.

Let $X_i^m$ be $X_i$ equipped with the action of $G_i\times \GGm$, where $\GGm$ acts trivially, and observe that the natural projection $Y_i\to X_i^m$ is $G_i\times \GGm$-equivariant. This map induces a bijection of the sets of $G_i\times \GGm$-orbits, of the sets of $B_i\times \GGm$-orbits, and by pull-back an inclusion of the groups of $B_i\times\GGm$-semiinvariant rational functions compatible with the maps $\rho$. It follows $S^p(Y_i)=S^p(X_i^m)$, $\A(Y_i)=\A(X_i^m)$, $\col_{Y_i}=\col_{X_i^m}$, and $\V_{Y_i}=\V_{X_i^m}$ for each $i \in \{1,2\}$.

Moreover, a generic stabilizer $H_{X_i^m}$ of $X_i^m$ is equal to $H_{Y_i}\cdot(\{e\}\times \GGm)$, where $H_{Y_i}$ is a generic stabilizer of $Y_i$. It follows that $H_{X_i^m}/H_{Y_i}$ is connected, whence $\Lambda(X_i^m)$ is a saturated sublattice of $\Lambda(Y_i)$ by \cite[Lemma 2.4]{gandini-spherorbclos}. Together with \cite[Theorem 4.4 and proof of Theorem 6.1]{knop-lv}, we obtain the equalities $\Sigma^{sc}(Y_i)=\Sigma^{sc}(X_i^m)$, and we have shown $\soc(Y_i)=\soc(X_i^m)$ for each $i \in \{1,2\}$. Since obviously $\soc(X_i^m)=\soc(X_i)$, the proof is complete.
\end{proof}

\begin{lemma}\label{lemma:quasivec}
Let $X$ be a simple spherical variety. If its socle is equal to the socle of a quasi-vectorial variety $Y$, then $X$ is quasi-vectorial.
\end{lemma}
\begin{proof}
Since $Y$ is quasi-vectorial then $|\col_Y|=|\col(Y)|$. By Remark ~\ref{rem:colors_variety_system} the spherical system of a spherical variety $Z$ determines the number of colors of $Z$. We deduce that $|\col_X|=|\col(X)|$, i.e.\ $X$ is quasi-vectorial.
\end{proof}

The following proposition is a key step in the proof of the smoothness criterion. We remark that it uses the uniqueness statement in the Luna Conjecture, which was proved by Losev in \cite{losev-uniqueness}. 
\begin{proposition}\label{prop:sphmod}
A simple spherical variety $X$ is $G$-isomorphic to a spherical module if and only if
\begin{enumerate}
\item \label{prop:sphmod:soc} its socle is the socle of a spherical module; and
\item \label{prop:sphmod:bas} the $|\D_X|$-tuple $(\rho_X(D))_{D\in\D_X}$ is a basis of $\Lambda(X)^*$.
\end{enumerate} 
\end{proposition}
\begin{proof}
Suppose that $X$ is a spherical module. Then any element $D\in\D_X$ has a global equation $f_D\in\CC[X]^{(B)}$, where $\CC[X]^{(B)}$ denotes theset of $B$-eigenvectors of $\CC[X]$. Its $B$-eigenvalue $\gamma_D$ belongs to $\Lambda(X)$, takes value $1$ on $\rho_X(D)$ and $0$ on $\rho_X(D')$ where $D'$ is any element of $\D_X$ different from $D$. This shows that both $(\rho_X(D))_{D\in\D_X}$ and $( \gamma_D)_{D\in\D_X}$ are linearly independent in $\lat(X)^*$ and $\lat(X)$, respectively.

Now part (\ref{prop:sphmod:bas}) follows if we show that $( \gamma_D)_{D\in\D_X}$ generates $\Lambda(X)$. Pick $\lambda\in\Lambda(X)$, choose $f\in\CC(X)^{(B)}_\lambda$, where $\CC(X)^{(B)}_\lambda$ denotes the set of $B$-eigenvectors of $B$-weight $\lambda$, and consider
\[
F = \prod_{D\in\D_X} f_D^{\langle \rho_X(D),\lambda\rangle}.
\]
Then $F$ is a $B$-semiinvariant rational function on $X$, with $\Div(F)=\Div(f)$ and $B$-eigenvalue belonging to the group generated by  $( \gamma_D)_{D\in\D_X}$. The quotient $F/f$ is then an invertible rational function on $X$, therefore constant. We deduce that $F$ and $f$ have the same $B$-eigenvalue, hence $\lambda$ is in the group generated by $( \gamma_D)_{D\in\D_X}$.

Now we show that the conditions (\ref{prop:sphmod:soc}) and (\ref{prop:sphmod:bas})  are sufficient for $X$ to be $G$-isomorphic to a spherical module. We may suppose that $G= (G,G)\times G^r$.

Let $E$ be a spherical module with the same socle as $X$, and let $E=\bigoplus_{i\in I} E_i$ be its decomposition into irreducibles as in Definition~\ref{def:module}. The $G$-action on $E$ is not uniquely determined by the socle; in particular the socle gives no information on the action of $G^r$ (see e.g.\ Lemma~\ref{lemma:socleindependent}). To prevent any difficulty arising from this fact, we let the bigger group $\widetilde G=G\times \GL(E)^G$ act on $E$ in the obvious way. We also let $\widetilde G$ act on $X$, by letting the $|I|$-dimensional torus $\GL(E)^G$ act trivially. Denote by $\widetilde C$ the radical $G^r\times \GL(E)^G$ of $\widetilde G$, by $\widetilde T$ the maximal torus $T\times \GL(E)^G$, and by $\widetilde B$ the Borel subgroup $B\times \GL(E)^G$ of $\widetilde G$. Notice that the assumptions of the proposition also apply to the $\widetilde G$-action, and proving it for $\widetilde G$ implies the proposition for $G$. In what follows, all invariants are now relative to the $\widetilde G$-action. We will write $\Sigma^{sc}$ for the two equal sets $\Sigma^{sc}(X)=\Sigma^{sc}(E)$, and since $\D_X$ is identified with $\D_E$ via a fixed bijection, we may sometimes write just $\D$ for both sets, and the same for $\col_X$ and $\V_X$.

For any $D\in\D_E$ (resp.\ $\D_X$), let $\gamma^E_D$ (resp.\ $\gamma^X_D$) be the element corresponding to $D$ in the basis of $\Lambda(E)$ (resp.\ $\Lambda(X)$) dual to $\rho_E(\D_E)$ (resp.\ $\rho_X(\D_X)$). Notice that this notation is compatible with Definition~\ref{def:module}. 
A rational function $f\in\CC(X)^{(\widetilde B)}$ with $\widetilde B$-eigenvalue $\gamma^X_D$ vanishes on $D$ and has no zero nor pole on any other $\widetilde B$-stable prime divisor, since $\langle \rho_X(D'),\gamma^X_D\rangle = \delta_{D',D}$ for all $D' \in \mathcal{D}_X$ and $\mathcal{D}_X$ is the set of all $B$-stable prime divisors on $X$ because $X$ is quasi-vectorial by Proposition~\ref{prop:colors_E} and Lemma~\ref{lemma:quasivec}. Consequently $f$ is a global equation of $D$. 

Write $\gamma^E_D = \eta^E_D + \epsilon^E_D$ and $\gamma^X_D =\eta^X_D + \epsilon^X_D$, where $\eta^E_D,\eta^X_D \in \Chi(\widetilde T\cap (\widetilde G,\widetilde G))$ and $\epsilon^E_D,\epsilon^X_D \in \Chi(\widetilde C)$.

Consider now the set $\D^1_E$ of Definition~\ref{def:module}. The restrictions $(\epsilon^E_D)|_{\GL(E)^G}$ for $D$ varying in $\D^1_E$ are by construction a basis of $\Chi(\GL(E)^G)$. It follows that the weights $\epsilon^E_D$ are a basis of the lattice $\Xi$ they generate and that the latter is saturated inside $\Chi(\widetilde C)$. Hence $\Xi$ is a direct summand of $\Chi(\widetilde C)$, and there exists a homomorphism $\Chi(\widetilde C)\to \Chi(\widetilde C)$ sending $\epsilon^E_D$ to $\epsilon^X_D$ for all $D\in\D^1_E$.

We extend the corresponding homomorphism $\widetilde C\to \widetilde C$ to $\widetilde G$ via the identity on $(\widetilde G,\widetilde G)$, and we denote the extension by $\phi\colon \widetilde G\to\widetilde G$. It also induces a homomorphism $\phi^*\colon\Chi(\widetilde T)\to\Chi(\widetilde T)$.

Now consider generic stabilizers $H_E\subseteq \widetilde G$ of $E$ and $H_X$ of $X$. The homogeneous spaces $\widetilde G/\overline H_X$ and $\widetilde G/\overline H_E$ have the same spherical system. By \cite[Theorem 1]{losev-uniqueness} we may assume that $\overline H_E=\overline H_X$.

It follows that the pull-back on $\widetilde G$ of any $D\in\col_X$ along $\widetilde G\to \widetilde G/H_X$ coincides with the pull-back of the corresponding color of $E$ along $\widetilde G\to \widetilde G/H_E$, because $D$ is the pull-back of a color of $\widetilde G/\overline H_X = \widetilde G/\overline H_E$ along $\widetilde G/H_X \to \widetilde G/\overline H_X$. Since $\gamma^X_D$ is the eigenvalue of global equation of $D$ in $X$ and $\gamma^E_D$ is the eigenvalue of a global equation of the corresponding color in $E$, the weights $\gamma^E_D$ and $\gamma^X_D$ are the $\widetilde B$-eigenvalues of two global equations in $\CC[\widetilde G]$ of this pull-back, hence they differ only by a character of $\widetilde C$. This shows that $\eta^X_D=\eta^E_D$ for all $D\in\col$.

On the other hand $D\in\V_X$ (resp.\ $\V_E$) is $\widetilde G$-stable and $\gamma^X_D$ (resp.\ $\gamma^E_D$) is the $\widetilde G$-eigenvalue of a global equation of $D$ in $X$ (resp.\ $E$).  It follows that $\gamma^X_D$ and $\gamma^E_D$ are $\widetilde G$-characters, and therefore $\eta^X_D=\eta^E_D=0$ for all $D\in\V$. 

At this point we have shown that $\phi^*(\gamma^E_D)=\gamma^X_D$ for all $D\in\D^1_E$. Since we know that $\phi^*$ is the identity on $\Sigma^{sc}$, it follows from Lemma~\ref{lemma:Sigma_N_D_1} and Remark~\ref{rem:sigma_N_sc_E} that $\phi^*(\Lambda(E))\subseteq \Lambda(X)$, hence $\phi^*$ also induces a dual map $\phi^{**}\colon \Lambda(X)^*\to\Lambda(E)^*$.

We compare now $\rho_E(D)$ and $\phi^{**}(\rho_X(D))$ for all $D\in\D$. They are equal on $\Sigma^{sc}$ by hypothesis, and
\[
\langle\phi^{**}(\rho_X(D)),\gamma^E_{D'}\rangle = \langle\rho_X(D),\phi^*(\gamma^E_{D'})\rangle = \langle\rho_X(D),\gamma^X_{D'}\rangle = \delta_{D,D'} = \langle\rho_E(D),\gamma^E_{D'}\rangle
\]
for all $D'\in\D^1_E$. This means that $\rho_E(D)$ and $\phi^{**}(\rho_X(D))$ coincide on a set of generators of $\Lambda(E)$, therefore they are equal for all $D\in\D$. Since $(\rho_E(D))_{D\in\D}$ is a basis of $\Lambda(E)^*$ and $(\rho_X(D))_{D\in\D}$ is a basis of $\Lambda(X)^*$, we also have that $\phi^{**}$ is an isomorphism. Consequently, so is $\phi^*|_{\Lambda(E)}\colon\Lambda(E)\to\Lambda(X)$.

We define a new action of $\widetilde G$ on $E$, denoting the obtained module by $E'$. Let $\psi\colon \widetilde G\to \GL(E)$ be the homomorphism induced by our original action, set $E'=E$ as vector spaces, and define $\psi'\colon \widetilde G\to \GL(E')$ to be $\psi' = \psi\circ\phi$. We claim that $E'$ is spherical, that $\Lambda(E')=\Lambda(X)$, $\V(E')=\V(X)$, and that $\D_{E'}$ can be identified with $\D_{X}$ compatibly with the maps $\rho_{E'}$ and $\rho_X$.

For the first claim, decompose
\[
\CC[E] = \bigoplus_{\lambda \in \wm(E)} V(\lambda)
\]
into a sum of irreducibles. Each $V(\lambda)$ is also an irreducible submodule of $\CC[E']$, of highest weight $\phi^*(\lambda)$. Since $\CC[E]$ is multiplicity free and $\phi^*$ is injective on $\Lambda(E)$, we deduce that $\CC[E']$ is also multiplicity free. This proves that $E'$ is spherical.

For the second claim, notice that $E'$ and $E$ have the same socle thanks to Lemma~\ref{lemma:socleindependent}; then, with similar considerations as for the first claim, we have equality of the lattices $\Lambda(E')=\Lambda(X)$. The equality $\rho_{E'}(D)=\rho_X(D)$ for all $D\in\D$ follows from the definition of $E'$.

By \cite[Theorem 1]{losev-uniqueness}, the open $\widetilde G$-orbits of $X$ and $E'$ are isomorphic. Finally, we apply \cite[Theorem 3.1]{knop-lv} to the simple varieties $X$ and $E'$ and we deduce that they are equivariantly isomorphic.
\end{proof}

\begin{corollary}\label{cor:quasivec}
For an affine spherical $G$-variety $X$ the following are equivalent:
\begin{enumerate}
\item\label{cor:quasivec:qv} $X$ is smooth and quasi-vectorial;
\item\label{cor:quasivec:GKV} $X$ is isomorphic to a product $G/K \times V$, where $K\subseteq G$ is a subgroup containing $(G,G)$ and $V$ is a spherical $G$-module;
\item\label{cor:quasivec:comb} the following two conditions hold:
\begin{enumerate}
\item\label{cor:quasivec:comb:soc} its socle is the socle of a spherical module,
\item\label{cor:quasivec:comb:bas}  the $|\D_X|$-tuple $(\rho_X(\D_X))_{D\in\D_X}$ can be completed to a basis of $\Lambda(X)^*$.
\end{enumerate}
\end{enumerate}
\end{corollary}
\begin{proof}
It is harmless to assume in the whole proof that $G = (G,G)\times G^r$.

(\ref{cor:quasivec:qv}) $\Rightarrow$ (\ref{cor:quasivec:GKV}).  Since $X$ is smooth and affine, it is isomorphic to a vector bundle
\[
X \cong G\times^K V
\]
on its closed $G$-orbit $G/K$, with fiber a $K$-module $V$, where $K$ is a reductive group and $V$ is spherical under the action of the connected component of $K$ containing the identity; see \cite[Corollary 2.2]{knop&bvs-classif}. The closed orbit $G/K$ has no colors, otherwise the inverse image in $X$ of one of its colors would be a color of $X$ not containing $G/K$. It follows from \cite[Proposition 2.4]{knop-asymptotic} that $K\supset (G,G)$, thus $K = (G,G)\times R$ where $R\subseteq G^r$. The inclusion of diagonalizable groups $R\to G^r$ has a right inverse, hence we may define an action of $G$ on $V$ extending that of $K$. Therefore $X\cong G/K\times V$.

(\ref{cor:quasivec:GKV}) $\Rightarrow$ (\ref{cor:quasivec:qv}). 
We only have to show that all colors of $V$ contain $0$, its unique closed $G$-orbit. This is Proposition~\ref{prop:colors_E}.  

(\ref{cor:quasivec:GKV}) $\Rightarrow$ (\ref{cor:quasivec:comb}). We know that $K = (G,G)\times R$ where $R\subseteq G^r$, and $X\cong(G^r/R)\times V$. Let us consider the action of $C\times G$ on $X$ where $C\cong G^r$ and acts only on the factor $G^r/R$ via its isomorphism with $G^r$, and $G$ acts only on the factor $V$. Denote the newly obtained $C\times G$-variety by $X'$.

Then $\soc(X)=\soc(X')$ by Lemma~\ref{lemma:socleindependent}. On the other hand $X'$ is of the form $X_1\times X_2$ where one factor of $C\times G$ acts only on $X_1$, and the other only on $X_2$. Therefore $\soc(X')=\soc(G^r/R)\times\soc(V)$. Since $C$ is abelian $\soc(G^r/R)$ is trivial, and $\soc(X')=\soc(V)$, which is condition (\ref{cor:quasivec:comb:soc}).

Then, notice that
\[
\CC[X] = \CC[G/K] \otimes_\CC \CC[V],
\]
which implies that $\Lambda(X) = \Lambda(G/K)\oplus \Lambda(V)$. Since $G/K$ is homogeneous and has no color, the map
\[
\begin{array}{ccc}
\D_V & \to & \D_X \\
D & \mapsto & G/K \times D
\end{array}
\]
is a bijection, and the element $\rho_X(G/K\times D)$ is equal to $\rho_V(D)$ on $\Lambda(V)$ and is zero on $\Lambda(G/K)$. Thanks to Proposition~\ref{prop:sphmod} the $|\D_V|$-tuple $(\rho_V(D))_{D\in\D_V}$ is a basis of $\Lambda(V)^*$, and it follows that $(\rho_X(D))_{D\in\D_X}$ can be completed to a basis of $\Lambda(X)^*$. This shows (\ref{cor:quasivec:comb:bas}).

(\ref{cor:quasivec:comb}) $\Rightarrow$ (\ref{cor:quasivec:GKV}). Since the socle of $X$ is that of a spherical module, the set $\col_X$ is the whole set of colors of $X$ by Lemma~\ref{lemma:quasivec} and Proposition~\ref{prop:colors_E}. For the same reason $\rho_X(\col_X)$ doesn't contain $0$.

Consider the vector subspace $N_1$ spanned by $\rho_X(\D_X)$ in the vector space $N=\Lambda(X)^*\otimes_\ZZ\QQ$. We claim that $N_1$ is generated as a convex cone by $\rho_X(\col_X)$ together with $N_1\cap \V(X)$. To show the claim, we observe first of all that $\rho_X(\V_X)$ is contained in $N_1\cap\V(X)$. Then we know that $\V(X)$ together with $\rho_X(\col(X))$ generates $N$ as a convex cone, thanks to \cite[Theorem 4.4]{knop-lv} applied to the $G$-equivariant map $X\to \{\text{pt}\}$. At this point the claim follows because $\V(X)$ is a convex cone.

Then, by \cite[Theorem 4.4]{knop-lv}, there exists a $G$-equivariant map $X_0\to Y$, where $X_0$ is the open $G$-orbit of $X$, the variety $Y$ is spherical and $G$-homogeneous, such that the pull-back of functions induces an identification of $\Lambda(Y)$ with the direct summand of $\Lambda(X)$ where $N_1$ is zero, and $\col_X$ is exactly the set of colors of $X_0$ mapped dominantly onto $Y$ (up to identifying the colors of $X$ and of $X_0$).

It follows that $Y$ has no color, hence $(G,G)$ acts trivially on $Y$ and $Y=G/K$ where $K\supseteq (G,G)$. Moreover, since $\rho_X(\D_X)\subseteq N_1$, by \cite[Theorem 4.1]{knop-lv} the map $X_0\to Y$ extends to a $G$-equivariant map $X\to Y$.

Now we can write $X$ as a $G$-equivariant bundle over $Y=G/K$:
\[
X = G\times^K V
\]
where $V$ is the fiber over $eK\in G/K$ of the map $X\to Y$. Since the inclusion $K\to G$ has a right inverse, the variety $V$ is also a $G$-variety and the bundle is trivial:
\[
X = G/K \times V.
\]
Hence $\Lambda(X)=\Lambda(G/K)\oplus \Lambda(V)$, and $N_1 = \Lambda(G/K)^\perp$. With the same argument as in the previous implication we have that $\soc(X)=\soc(V)$ and that $(\rho_V(D))_{D\in\D_V}$ is identified with $(\rho_X(D))_{D\in\D_X}$ and is a basis of $\Lambda(V)^*$.

Now $V$ satisfies the hypotheses of Proposition~\ref{prop:sphmod}, so it is a spherical $G$-module and the proof is complete.
\end{proof}

In order to apply Corollary~\ref{cor:quasivec} to a general affine spherical variety, we need to introduce {\em localization} of spherical varieties.

\begin{definition} \label{def:socloc}
Let
\[
\soc(X)=(S^p(X), \Sigma^{sc}(X), \A(X), \col_X, \V_X, \rho'_X\colon\D_X\to (\Z\Sigma^{sc}(X))^*)
\]
be the socle of a simple spherical $G$-variety $X$, and let $S'\subseteq S$. The \textbf{ localization of $\soc(X)$ at $S'$} is defined as follows:
\[
\soc(X)_{S'}=(S^p(X)_{S'}, \Sigma^{sc}(X)_{S'}, \A(X)_{S'}, \col_{X,S'}, \V_{X,S'}, \rho'_{X,S'})
\]
where
\begin{enumerate}
\item $S^p(X)_{S'} = S^p(X)\cap S'$,
\item $\Sigma^{sc}(X)_{S'}=\Sigma^{sc}(X)\cap \Z S'$,
\item $\A(X)_{S'}= \bigcup_{\alpha\in S'\cap \Sigma(X)} \A(X,\alpha)$,
\item $\rho'_{X,S'}$ is the restriction of $\rho'_X$ to $\Z \Sigma^{sc}(X)_{S'}$,
\item $\col_{X,S'}$ is the set of colors of the spherical system $(S^p(X)_{S'}, \Sigma^{sc}(X)_{S'}, \A(X)_{S'})$ (notice that $\col_{X,S'}$ is naturally a subset of the set of all colors of $X$),
\item $\V_{X,S'} = \V_X\cup (\col_X\smallsetminus \col_{X,S'})$.
\end{enumerate}
\end{definition}

We recall the local structure theorem for spherical varieties, see e.g.~\cite[Theorem 2.3 and Proposition 2.4]{knop-asymptotic}.

\begin{theorem}\label{thm:local}
Let $X$ be a spherical variety and $Y\subseteq X$ a $G$-orbit. Let $P$ be the stabilizer of all colors of $X$ not containing $Y$, and let $L$ be a Levi subgroup of $P$. Then there exists an affine, $L$-stable and $L$-spherical, locally closed subvariety $Z$ of $X$ such that
\[
\begin{array}{ccc}
P^u\times Z & \to &  X_{Y,B} \\
(p,z) &\mapsto & pz
\end{array}
\]
is an isomorphism, where $X_{Y,B}$ is the open subset of $X$ defined as
\[
X_{Y,B} = \{ x\in X \;|\; \overline{Bx}\supseteq Y\}.
\]
\end{theorem}

\begin{definition}
Let $X$ be a spherical $G$-variety and $Y\subseteq X$ be a $G$-orbit. We define
\[
X_{Y,G}=\{x\in X\;|\; \overline{Gx}\supseteq Y\}.
\]
\end{definition}

Notice that $Y$ is the unique closed $G$-orbit of $X_{Y,G}$, which is open in $X$ and $G$-stable.

\begin{proposition}\label{prop:socloc}
Let $X$ be a spherical $G$-variety and $Y\subseteq X$ a $G$-orbit. Let $L$ and $Z$ be as in Theorem~\ref{thm:local} with the additional assumption that $L$ contains $T$, and let $S'\subseteq S$ be the set of simple roots of $L$. Then the socle of $Z$ as a spherical $L$-variety is the localization of $\soc(X_{Y,G})$ at $S'$. In particular, $Z$ is quasi-vectorial.
\end{proposition}
\begin{proof}
The proposition follows from \cite[Lemma 3.5.5]{losev-uniqueness}, where we set $\mathcal D'$ of loc.cit.\ equal to the set of colors of $X$ not containing $Y$.
\end{proof}

We come to the smoothness criterion. We remark that spherical modules were classified in \cite{kac-rmks,brion-repexc,benson-ratcliff-mf, leahy}; see also~\cite{knop-rmks}. Their socles can be deduced from the list in Table~\ref{table:socles}, thanks to Theorem~\ref{thm:soclesmodules} below. The same list is found in \cite{gagliardi-camus}, where {\em Luna diagrams} are used to denote spherical systems.

\begin{theorem}\label{thm:camus_smoothnes_crit}
Let $X$ be a spherical $G$-variety and $Y\subseteq X$ be a $G$-orbit. Then $X$ is smooth in all points of $Y$ if and only if
\begin{enumerate}
\item the localization of $\soc(X_{Y,G})$ at $S'$ (with the notations of Proposition~\ref{prop:socloc}) is the socle of a spherical module.\label{thm:camus_smoothness_crit:socle}
\item the $|\D_{X_{Y,G}}|$-tuple $(\rho_X(D))_{D\in\D_{X_{Y,G}}}$ can be completed to a basis of $\Lambda(X)^*$. \label{thm:camus_smoothness_crit:basis}
\end{enumerate}
\end{theorem}
\begin{proof}
Let $Z$ be as in Proposition~\ref{prop:socloc}. The smoothness of $X$ along $Y$ is equivalent to the smoothness of $Z$, which is an affine simple spherical variety whose socle is the localization of $\soc(X_{Y,G})$ at $S'$ by Proposition~\ref{prop:socloc}. The theorem now follows  from Corollary~\ref{cor:quasivec}.
\end{proof}

\begin{remark}
Example~\ref{ex:need_tuple} in Section~\ref{sec:smoothwm} below shows that it is not possible to replace, in Theorem~\ref{thm:camus_smoothnes_crit}, the $|\D_{X_{Y,G}}|$-tuple $(\rho_X(D))_{D\in\D_{X_{Y,G}}}$ with the {\em set} $\rho_X(\D_{X_{Y,G}})$.
\end{remark}

Finally, we relate the socle of any spherical module to those in Table~\ref{table:socles}. First we recall some definitions.

\begin{definition}[\cite{camus}]
Consider the socles $\soc_i=(S^p_i,\Sigma^{sc}_i,\A_i,\Delta_i,\V_i,\rho'_i\colon \D_i\to(\ZZ\Sigma_i^{sc})^*)$ for $i\in\{1,2\}$ of two simple spherical $G_i$-varieties. They are \textbf{isomorphic} if they are equal up to an isomorphism $\varphi$ of the Dynkin diagrams of $G_1$ and $G_2$, i.e.\ if $S^p_2=\varphi(S^p_1)$ and $\varphi(\Sigma^{sc}_1)=\Sigma^{sc}_2$ (where we have extended $\varphi$ to a map between the two root lattices), and if $\A_1$, $\Delta_1$, $\V_1$ can be identified resp.\ with $\A_2$, $\Delta_2$, $\V_2$ in such a way that $\langle\rho'_1(D), \sigma\rangle = \langle\rho'_2(D),\varphi(\sigma)\rangle$ for all $D\in \A_1\cup\Delta_1\cup V_1$ and all $\sigma\in\Sigma^{sc}_1$.
\end{definition}

Notice that the above definition includes the case where $G_1=G_2$ and $\varphi$ is an automorphism of its Dynkin diagram.

\begin{definition}[{\cite[Section~5]{knop-rmks}}]\label{def:geomequiv}
Two representations $\eta_i\colon G_i\to \GL(V_i)$ for $i\in\{1,2\}$ are \textbf{geometrically equivalent} if there is an isomorphism $\Psi\colon V_1\to V_2$ inducing the isomorphism $\GL(\Psi)\colon \GL(V_1)\to \GL(V_2)$ such that $\GL(\Psi)(\eta_1(G_1))=\eta_2(G_2)$.
\end{definition}

\begin{lemma}\label{lemma:socleiso}
Let $\eta_i\colon G_i\to \GL(V_i)$ for $i\in\{1,2\}$ be two geometrically equivalent and spherical representations, such that no simple normal subgroup of $G_i$ acts trivially on $V_i$. Then the socles of $V_1$ and of $V_2$ are isomorphic.
\end{lemma}
\begin{proof}
Fix a map $\Psi$ as in Definition~\ref{def:geomequiv}. Using $\Psi$ we may identify $V_1$ and $V_2$ as vector spaces, denoting them both $V$. This yields $\eta_1(G_1)=\eta_2(G_2)$ as subgroups of $\GL(V)$; denote them both by $G$.

Choose a Borel subgroup $B_1$ and a maximal torus $T_1\subseteq B_1$ of $G_1$. We fix the Borel subgroup $\eta_1(B_1)$ and the maximal torus $\eta_1(T_1)$ of $G$, and we fix the Borel subgroup $\eta_2^{-1}(\eta_1(B_1))$ and the maximal torus $\eta_2^{-1}(\eta_1(T_1))$ of $G_2$ (these groups are connected because the kernel of $\eta_2$ is central).

Since no simple normal subgroup of $G_i$ acts trivially on $V$, we obtain an identification of the Dynkin diagrams of $G_1$, $G$, and $G_2$. Moreover, for all $i\in\{1,2\}$ the socle $\soc(V)$ defined with respect to the action of $G_i$ is equal to the socle defined with respect to the action of $\psi_i(G_i)$.

At this point, considered under the action of $G$, the modules $V_1$ and $V_2$ are the same spherical module under the action of the same group $G$, thus they have the same socle. Considered under the action of $G_i$, their socles are equal up to the above identification of the Dynkin diagrams, which finishes the proof.
\end{proof}

\begin{remark}
Observe that the converse to Lemma~\ref{lemma:socleiso} does not hold. For example, the standard actions of $\SL(n)$ and $\GL(n)$ on $\C^n$ have isomorphic socles but are not geometrically equivalent.
\end{remark}

\begin{theorem}\label{thm:soclesmodules}
Let $V$ be a spherical $G$-module. Then its socle is, up to isomorphism, a product of socles of Table~\ref{table:socles}.
\end{theorem}
\begin{proof}
If $G$ has some simple normal subgroup $G'$ (with set of simple roots $S'$) acting trivially on $V$, then $\soc(V)$ is the product of the ``trivial'' socle $(S',\emptyset, \emptyset, \emptyset, \emptyset,\emptyset)$ and the socle of $V$ under the action of $G/G'$, hence we may assume that no such $G'$ exists.

Suppose that $G=G_1\times \ldots\times G_n$ and $V=V_1\oplus\ldots\oplus V_n$ such that for all $i\in\{1,\ldots,n\}$ the factor $G_i$ acts non-trivially on the $i$-th summand $V_i$ and trivially on the other summands. Then the socle of $V$ under the action of $G$ is the product of the socles of the modules $V_1,\ldots,V_n$ under the action of resp.\ $G_1,\ldots,G_n$.

Thanks to Lemma~\ref{lemma:socleiso}, we may prove the theorem assuming that there is no such decomposition, not even up to geometric equivalence. We recall that this is the definition of an \textbf{indecomposable} module in the sense of \cite[Section~5]{knop-rmks}.

The module $V$ might be a reducible $G$-module, so we denote by $V=W_1\oplus\ldots\oplus W_m$ the decomposition into irreducible summands, unique by sphericity of $V$.

Extend the $G$-action on $V$ to the group $G_0=G\times (\CC^*)^m$ by letting the $i$-th $\CC^*$-factor act on $W_i$ by multiplication. The socle of $V$ with respect to the action of $G$ is equal to the socle with respect to the action of $G_0$. Denote by $\psi\colon G_0\to\GL(V)$ this representation: then the center of $\psi(G_0)$ has dimension equal to $m$. We recall that this is the definition of a \textbf{saturated} module in the sense of \cite[Section~5]{knop-rmks}.

By \cite[Theorem~5.1]{knop-rmks}, the module $V$ up to geometric equivalence appears in the list of modules in \cite[Section~5]{knop-rmks}. Their socles are given in Table~\ref{table:socles}. We remark that for the modules $(\CC^m\otimes \CC^n)\oplus\CC^n$ and $(\CC^m\otimes \CC^n)\oplus(\CC^n)^*$ under the action of $\GL(m)\times \GL(n)$ we have assumed $m\geq 1, n\geq 2$, which is a larger
set of indices than that given in \cite[Section~5]{knop-rmks}. Knop communicated the revised range of indices for these modules to us.
\end{proof}

We explain the notations of Table~\ref{table:socles}.

For each module $V$ under the action of the group $G$ we give the semisimple type of $G$ and the sets $S^p(V)$ and $\Sigma^{sc}(V)$. Here the simple roots are denoted by $\alpha_1,\alpha_2,\ldots$ and numbered as in \cite{bourbaki-geadl47}; we use the notation $\alpha_1'$, $\alpha_1''$, $\ldots$ if the Dynkin diagram of $G$ has more than one connected component.

Then, whenever $\A(V)\neq\emptyset$, instead of specifying the whole set $\A(V)$ we give the values of $\rho'(D)$ on all spherical roots for only one color $D\in \A(V)$. The whole set $\A(V)$ is the unique one containing $D$ and such that $(S^p(V), \Sigma^{sc}(V),\A(V))$ is a spherical system. The set $\col_V$ is equal to the whole set of colors $\col(V)$ of the spherical system $(S^p(V),\Sigma^{sc}(V),\A(V))$, and is described in Definition~\ref{def:colors_sph_system}.

Finally, we describe the elements of $\V_V$ and their values on the spherical roots as follows. In each case, either $\V_V$ is empty, or it contains exactly one element $\nu$, or it contains exactly two elements. If $\V_V=\{\nu\}$, then either $\Sigma^{sc}(V)$ is empty, or there exists a spherical root $\gamma\in\Sigma^{sc}(V)$ such that $\rho'(\nu)(\gamma)=-1$, and $\rho'(\nu)(\sigma)=0$ for all $\sigma\in\Sigma^{sc}(V)\smallsetminus\{\gamma\}$. If $\Sigma^{sc}(V)$ is empty then $\rho'(\nu)$ is the empty map, and we report $\V_V$ only as $\{\nu\}$. Otherwise we denote $\nu$ by $-\gamma^*$ and we report $\V_V$ as $\{-\gamma^*\}$. If $\Sigma^{sc}(V)$ contains more than one element we indicate $\gamma\in\Sigma^{sc}(V)$ explicitly, otherwise $\gamma$ is obviously the unique element of $\Sigma^{sc}(V)$.

If $\V_V$ contains two elements $\nu_1,\nu_2$, then $\Sigma^{sc}(V)$ contains two elements $\gamma_1,\gamma_2$ such that for all $i\in\{1,2\}$ we have $\langle\rho'(\nu_i),\gamma_i\rangle=-1$ and $\langle\rho'(\nu_i),\sigma\rangle=0$ for all spherical root $\sigma$ different from $\gamma_i$. In this case we report $\V_V$ as $\{-\gamma_1^*,-\gamma_2^*\}$, and indicate $\gamma_1,\gamma_2\in\Sigma^{sc}(V)$ explicitly.

\begin{longtable}{|c|c|c|c|c|}
\caption{Socles of spherical modules.}\label{table:socles}\\
\hline
$G$ & $S^{p}(V)$ & $\Sigma^{sc}(V)$ & $\A(V)$ & $\V_V$  \\
\hline
Any%
\footnote{This is the socle of the module $\{0\}$.} & $S$ & $\emptyset$ & $\emptyset$ & $\emptyset$ \\
\hline
Torus%
\footnote{This is the socle of a one-dimensional module.} & $\emptyset$ & $\emptyset$ & $\emptyset$ & $\nu$ \\
\hline
\end{longtable}

\begin{longtable}{|c|c|c|c|c|}
\hline
$G$ & $S^{p}(V)$ & $\Sigma^{sc}(V)$ & $\A(V)$ & $\V_V$  \\
\hline
$\sA_n$, $n\geq 2$ & $\emptyset$ & $2\alpha_1,\ldots,2\alpha_{n-1},\gamma=2\alpha_n$ & $\emptyset$ & $-\gamma^*$  \\
\hline
\begin{tabular}{c} $\sA_n\times \sA_n$ \\ $n\geq 1$ \end{tabular} & $\emptyset$ & \begin{tabular}{c} $\alpha_1+\alpha_1',\ldots,\alpha_{n-1}+\alpha_{n-1}'$ \\ $\gamma=\alpha_n+\alpha_n'$ \end{tabular} & $\emptyset$ & $-\gamma^*$  \\
\hline
$\sA_n$, $n\geq 3$ & $\emptyset$ & \begin{tabular}{c} $\alpha_1+\alpha_2$, $\alpha_2+\alpha_3$,\\ $\ldots$,$\gamma=\alpha_{n-1}+\alpha_n$ \end{tabular} & $\emptyset$ & $-\gamma^*$  \\
\hline
$\sA_n$, $n\geq 5$ odd & $\emptyset$ & \begin{tabular}{c} $\alpha_1+\alpha_2$, $\alpha_2+\alpha_3$, $\ldots$,\\ $\gamma=\alpha_{n-2}+\alpha_{n-1}$,\\ $\alpha_{n-1}+\alpha_n$ \end{tabular} & $\emptyset$ & $-\gamma^*$  \\
\hline
\end{longtable}

\begin{longtable}{|c|c|c|c|c|}
\hline
$G$ & $S^{p}(V)$ & $\Sigma^{sc}(V)$ & $\A(V)$ & $\V_V$  \\
\hline
$\sA_3\times\sC_2$ & $\emptyset$ & $S, \gamma=\alpha_3$ & $1,0,-1,1,-1$ & $-\gamma^*$  \\
\hline
$\sA_1$ & $\emptyset$ & $\alpha_1$ & $1$ & $-\gamma^*$  \\
\hline
$\sA_n$, $n\geq 4$ even & $\emptyset$ & \begin{tabular}{c} $\alpha_1+\alpha_2$, $\alpha_2+\alpha_3$, $\ldots$,\\ $\alpha_{n-2}+\alpha_{n-1}$, $\alpha_n$ \end{tabular} & $0,\ldots,0,1$ & $\emptyset$  \\
\hline
\begin{tabular}{c} $\sA_n\times\sA_{n+1}$ \\ $n\geq1$ \end{tabular} & $\emptyset$ & \begin{tabular}{c} $\alpha_1,\ldots,\alpha_n,$ \\ $\alpha'_1,\ldots,\gamma=\alpha'_{n+1}$ \end{tabular} & \begin{tabular}{c} $1,-1,0,\ldots,0,$ \\ $-1,1,0,\ldots,0$ \end{tabular} & $-\gamma^*$  \\
\hline
\begin{tabular}{c} $\sA_n\times\sA_n$ \\ $n\geq1$ \end{tabular} & $\emptyset$ & \begin{tabular}{c} $\alpha_1,\ldots,\gamma=\alpha_n,$ \\ $\alpha'_1,\ldots,\alpha'_n$ \end{tabular} & \begin{tabular}{c} $1,-1,0,\ldots,0,$ \\ $-1,1,0,\ldots,0$ \end{tabular} & $-\gamma^*$  \\
\hline
\begin{tabular}{c} $\sA_n\times\sA_n$ \\ $n\geq1$ \end{tabular} & $\emptyset$ & \begin{tabular}{c} $\alpha_1,\ldots,\alpha_n,$ \\ $\alpha'_1,\ldots,\gamma=\alpha'_n$ \\ \end{tabular} & \begin{tabular}{c} $1,0,\ldots,0,$ \\ $1,-1,0,\ldots,0$ \end{tabular} & $-\gamma^*$  \\
\hline
$\sA_1\times\sA_1\times\sA_1$ & $\emptyset$ & $\gamma_1=\alpha_1,\alpha'_1,\gamma_2=\alpha''_1$ & $1,1,-1$ & \begin{tabular}{c} $-\gamma_1^*$, \\ $-\gamma_2^*$ \end{tabular}  \\
\hline
\end{longtable}

\begin{longtable}{|c|c|c|c|c|}
\hline
\multicolumn{5}{|c|}{Notation: $\alpha_{a,b}=\alpha_a+\alpha_{a+1}+\ldots+\alpha_{b-1}+\alpha_b$}  \\
\hline
$G$ & $S^{p}(V)$ & $\Sigma^{sc}(V)$ & $\A(V)$ & $\V_V$  \\
\hline
$\sA_n$, $n\geq 1$ & $\alpha_2,\ldots,\alpha_n$ & $\emptyset$ & $\emptyset$ & $\emptyset$ \\
\hline
$\sC_n$, $n\geq 2$ & $\alpha_2,\ldots,\alpha_n$ & $\emptyset$ & $\emptyset$ & $\emptyset$  \\
\hline
$\sB_n$, $n\geq 2$ & $\alpha_2,\ldots,\alpha_n$ & $2\alpha_{1,n}$ & $\emptyset$ & $-\gamma^*$  \\
\hline
$\sD_n$, $n\geq 3$ & $\alpha_2,\ldots,\alpha_n$ &
$2\alpha_{1,n-2}+\alpha_{n-1}+\alpha_n$
& $\emptyset$ & $-\gamma^*$  \\
\hline
\begin{tabular}{c} $\sA_n$ \\ $n\geq 4$ even \end{tabular} & \begin{tabular}{c}$\alpha_1,\alpha_3,\ldots,$\\ $\alpha_{n-3},\alpha_{n-1}$ \end{tabular} &  \begin{tabular}{c} $\alpha_1+2\alpha_2+\alpha_3$, \\ $\alpha_3+2\alpha_4+\alpha_5$,\\ $\ldots$,\\ $\alpha_{n-3}+2\alpha_{n-2}+\alpha_{n-1}$\end{tabular} & $\emptyset$ & $\emptyset$  \\
\hline
\begin{tabular}{c} $\sA_n$ \\ $n\geq 3$ odd \end{tabular} & \begin{tabular}{c}$\alpha_1,\alpha_3,$ \\ $\alpha_{n-2}\ldots,\alpha_n$\end{tabular} &  \begin{tabular}{c} $\alpha_1+2\alpha_2+\alpha_3$, \\ $\alpha_3+2\alpha_4+\alpha_5$,\\ $\ldots$,\\ $\gamma=\alpha_{n-2}+2\alpha_{n-1}+\alpha_{n}$\end{tabular} & $\emptyset$ & $-\gamma^*$  \\
\hline
\begin{tabular}{c} $\sA_n\times \sA_m$ \\ $m>n\geq 1$ \end{tabular} & $\alpha'_{n+1},\ldots,\alpha'_m$ & $\alpha_1+\alpha_1',\ldots,\alpha_n+\alpha_n'$  & $\emptyset$ & $\emptyset$  \\
\hline
\begin{tabular}{c} $\sA_1\times \sC_m$ \\$m\geq 2$ \end{tabular} & $\alpha'_3,\ldots,\alpha'_m$ & \begin{tabular}{c} $\alpha_1+\alpha_1'$,\\ $\gamma=\alpha_1'+2\alpha'_{2,m-1}+\alpha'_m$
\end{tabular} & $\emptyset$ & $-\gamma^*$  \\
\hline
$\sB_3$ & $\alpha_1,\alpha_2$ & $\alpha_1+2\alpha_2+3\alpha_3$ & $\emptyset$ & $-\gamma^*$  \\
\hline
$\sB_4$ & $\alpha_2,\alpha_3$ & \begin{tabular}{c} $\gamma=\alpha_1+\alpha_2+\alpha_3+\alpha_4$, \\ $\alpha_2+2\alpha_3+3\alpha_4$ \end{tabular} & $\emptyset$ & $-\gamma^*$  \\
\hline
$\sD_5$ & $\alpha_2,\alpha_3,\alpha_4$ & $\alpha_2+2\alpha_3+\alpha_4+2\alpha_5$ & $\emptyset$ & $\emptyset$  \\
\hline
$\sG_2$ & $\alpha_2$ & $4\alpha_1+2\alpha_2$ & $\emptyset$ & $-\gamma^*$  \\
\hline
$\sE_6$ & $\alpha_2,\alpha_3,\alpha_4,\alpha_5$ & \begin{tabular}{c}
$2\alpha_1+\alpha_2+2\alpha_3+2\alpha_4+\alpha_5$,\\$\gamma=\alpha_2+\alpha_3+2\alpha_{4,6}$
\end{tabular} & $\emptyset$ & $-\gamma^*$  \\
\hline
$\sD_4$ & $\alpha_2$ & \begin{tabular}{c} $\gamma_1=\alpha_1+\alpha_2+\alpha_3$, \\ $\gamma_2=\alpha_1+\alpha_2+\alpha_4$,\\ $\alpha_2+\alpha_3+\alpha_4$ \end{tabular} & $\emptyset$ & \begin{tabular}{c} $-\gamma_1^*$, \\ $-\gamma_2^*$ \end{tabular}  \\
\hline
$\sA_n$, $n\geq 2$ & $\alpha_2,\ldots,\alpha_{n-1}$ & $\alpha_{1,n}$ & $\emptyset$ & $-\gamma^*$  \\
\hline
\end{longtable}

\begin{longtable}{|c|c|c|c|c|}
\hline
\multicolumn{5}{|c|}{Notation: $\alpha_{a,b}=\alpha_a+\alpha_{a+1}+\ldots+\alpha_{b-1}+\alpha_b$}  \\
\hline
$G$ & $S^{p}(V)$ & $\Sigma^{sc}(V)$ & $\A(V)$ & $\V_V$  \\
\hline
\begin{tabular}{c} $\sA_2\times\sC_m$ \\ $m\geq3$ \end{tabular} & $\alpha_4',\ldots,\alpha_m'$ & \begin{tabular}{c} $\alpha_1,\alpha_2,\alpha_1',\alpha_2'$, \\ $\alpha'_2+2\alpha'_{3,m-1}+\alpha'_m$ \end{tabular}  & \begin{tabular}{c} $1,0,1,-1,$ \\ $0$ \end{tabular} & $\emptyset$  \\
\hline
\begin{tabular}{c} $\sA_n\times\sC_2$ \\ $n\geq 4$ \end{tabular} & $\alpha_5,\ldots,\alpha_n$ & \begin{tabular}{c} $\alpha_1,\alpha_2,\alpha_3,$ \\ $\alpha'_1,\alpha'_2$ \end{tabular} & \begin{tabular}{c} $1,0,-1,$ \\ $1,-1$ \end{tabular} & $\emptyset$  \\
\hline
$\sA_n$, $n\geq 2$ & $\alpha_3,\ldots,\alpha_n$ & $\alpha_1$ & $1$ & $\emptyset$  \\
\hline
\begin{tabular}{c} $\sA_n\times\sA_m$ \\ $m-2\geq n\geq1$ \end{tabular} & \begin{tabular}{c} $\alpha'_{n+3},\ldots,$ \\ $\alpha'_m$ \end{tabular} & \begin{tabular}{c} $\alpha_1,\ldots,\alpha_n,$ \\ $\alpha'_1,\ldots,\alpha'_{n+1}$ \end{tabular} & \begin{tabular}{c} $1,-1,0,\ldots,0,$ \\ $-1,1,0,\ldots,0$ \end{tabular} & $\emptyset$  \\
\hline
\begin{tabular}{c} $\sA_n\times\sA_m$ \\ $n>m\geq1$ \end{tabular} & \begin{tabular}{c} $\alpha_{m+2},\ldots,$ \\ $\alpha_n$ \end{tabular} & \begin{tabular}{c} $\alpha_1,\ldots,\alpha_m,$ \\ $\alpha'_1,\ldots,\alpha'_m$ \end{tabular} & \begin{tabular}{c} $1,-1,0,\ldots,0,$ \\ $-1,1,0,\ldots,0$ \end{tabular}   & $\emptyset$  \\
\hline
\begin{tabular}{c} $\sA_n\times\sA_m$ \\ $n>m\geq1$ \end{tabular} & \begin{tabular}{c} $\alpha_{m+2},\ldots,$ \\ $\alpha_{n-1}$ \end{tabular} & \begin{tabular}{c} $\alpha_1,\ldots,\alpha_m,$ \\ $\alpha'_1,\ldots,\alpha'_m$, \\ $\alpha_{m+1,n}$ \end{tabular} & \begin{tabular}{c} $1,0,\ldots,0,$ \\ $1,-1,0,\ldots,0$, \\ $0$ \end{tabular} & $\emptyset$  \\
\hline
\begin{tabular}{c} $\sA_n\times\sA_m$ \\ $m>n\geq1$ \end{tabular} & \begin{tabular}{c} $\alpha'_{n+2},\ldots,$ \\ $\alpha'_m$ \end{tabular} & \begin{tabular}{c} $\alpha_1,\ldots,\alpha_n,$ \\ $\alpha'_1,\ldots,\alpha'_n$ \end{tabular} & \begin{tabular}{c} $1,-1,0,\ldots,0,$ \\ $-1,1,0,\ldots,0$ \end{tabular} & $\emptyset$  \\
\hline
\begin{tabular}{c} $\sA_n\times\sA_1\times\sA_1$ \\ $n\geq2$ \end{tabular} & $\alpha_3,\ldots,\alpha_n$ & $\alpha_1,\alpha'_1,\gamma=\alpha''_1$ & $1,1,-1$ & $-\gamma^*$  \\
\hline
\begin{tabular}{c} $\sA_n\times\sA_1\times\sA_m$ \\ $n,m\geq2$ \end{tabular} & \begin{tabular}{c} $\alpha_3,\ldots,\alpha_n,$ \\ $\alpha''_3,\ldots,\alpha''_m$ \end{tabular} & $\alpha_1,\alpha'_1,\alpha''_1$ & $1,1,-1$ & $\emptyset$  \\
\hline
$\sC_n$, $n\geq 2$ & $\alpha_3,\ldots,\alpha_n$ &  \begin{tabular}{c} $\alpha_1,$ \\ $\gamma=\alpha_1+2\alpha_{2,n-1}+\alpha_n$ \end{tabular}  & $1,0$ & $-\gamma^*$  \\
\hline
\begin{tabular}{c} $\sC_n\times\sA_1$ \\ $n\geq 2$ \end{tabular} & $\alpha_3,\ldots,\alpha_n$ & \begin{tabular}{c} $\alpha_1, \alpha'_1,$ \\ $\gamma=\alpha_1+2\alpha_{2,n-1}+\alpha_n$ \end{tabular} & \begin{tabular}{c} $1,1,$ \\ $0$ \end{tabular} & $-\gamma^*$  \\
\hline
\begin{tabular}{c} $\sC_n\times\sA_1\times\sA_1$ \\ $n\geq 2$ \end{tabular} & $\alpha_3,\ldots,\alpha_n$ & \begin{tabular}{c} $\alpha_1, \alpha'_1,\gamma_1=\alpha''_1,$ \\ $\gamma_2=\alpha_1+2\alpha_{2,n-1}+\alpha_n$ \end{tabular} & \begin{tabular}{c} $1,1,-1,$ \\ $0$ \end{tabular} & \begin{tabular}{c} $-\gamma_1^*$, \\ $-\gamma_2^*$ \end{tabular}  \\
\hline
\begin{tabular}{c} $\sC_n\times\sA_1\times\sA_m$ \\ $n,m\geq 2$ \end{tabular} & \begin{tabular}{c} $\alpha_3,\ldots,\alpha_n,$ \\ $\alpha''_3,\ldots,\alpha''_m$ \end{tabular} & \begin{tabular}{c} $\alpha_1, \alpha'_1,\alpha''_1,$ \\ $\gamma=\alpha_1+2\alpha_{2,n-1}+\alpha_n$ \end{tabular}  & \begin{tabular}{c} $1,1,-1,$ \\ $0$ \end{tabular} & $-\gamma^*$  \\
\hline
\begin{tabular}{c} $\sC_n\times\sA_1\times\sC_m$ \\ $n,m\geq 2$ \end{tabular} & \begin{tabular}{c} $\alpha_3,\ldots,\alpha_n,$ \\ $\alpha''_3,\ldots,\alpha''_m$ \end{tabular} & \begin{tabular}{c} $\alpha_1, \alpha'_1,\alpha''_1,$ \\ $\gamma_1=\alpha_1+2\alpha_{2,n-1}+\alpha_n,$ \\ $\gamma_2=\alpha''_1+2\alpha''_{2,n-1}+\alpha''_n$ \end{tabular} & \begin{tabular}{c} $1,1,-1,$ \\ $0,$ \\ $0$ \end{tabular} & \begin{tabular}{c} $-\gamma_1^*$, \\ $-\gamma_2^*$ \end{tabular}  \\
\hline
\end{longtable}

\section{Combinatorial characterization of smooth weight monoids} \label{sec:smoothwm}

In this section we state and prove the main result of this paper. In order to do so we introduce a few more notions.

\begin{definition} \label{def:thminvariants}
Let $\wm$ be a normal submonoid of $\dw$, and let $\Sigma$ be a subset of $\Sigma^{sc}(G)$ that is adapted to $\wm$. We define the following:
\begin{enumerate}[1.]
\item $\s(\wm, \Sigma) = (S^p(\wm), \Sigma, \A(\wm, \Sigma))$ is the spherical system constructed in the proof of Proposition~\ref{prop:luna-adapt}, $\col(\Sigma, \wm)$ is the set of colors of this spherical system, and $(\Z\wm,c)$ is the augmentation constructed in the same proposition.
\item $\V(\wm,\Sigma):=\{v\in\Hom_{\Z}(\Z\Gamma,\Q): \langle
  v,\sigma\rangle\leq0\mbox{ for all }\sigma\in\Sigma\}$.
\item $\mathcal{C}(\wm,\Sigma)$ is the maximal face of $\wm^\vee$ whose relative interior meets $\V(\wm,\Sigma)$.
\item $\mathcal{F}(\wm,\Sigma):=\{D \in \col(\wm,\Sigma) \colon c(D,\cdot) \in \mathcal{C}(\wm,\Sigma)\}$ 
\item $\mathcal{B}(\wm,\Sigma)$ is the set the primitive elements in $(\Z\wm)^*$ that lie on extremal rays of  $\mathcal{C}(\wm,\Sigma)$ which do not contain any element of $\{c(D,\cdot): D \in \mathcal{F}(\wm,\Sigma)\}$.
\item $\D(\wm,\Sigma):=\mathcal{F}(\wm,\Sigma) \cup \mathcal{B}(\wm,\Sigma)$.
\item $S(\wm, \Sigma):=\{\alpha \in \sr \colon \text{ $\alpha$ does not move any color in $\col(\Sigma,\wm)\setminus \mathcal{F}(\wm,\Sigma)$}\}$.
\item $\rho \colon \D(\wm,\Sigma) \to \Z\wm^*$ is defined by $\rho(D) = c(D,\cdot)$ if $D \in\mathcal{F}(\wm,\Sigma)$ and $\rho(D) = D$ for $D \in \mathcal{B}(\wm,\Sigma)$.
\item $\soc(\wm,\Sigma):=(S^p(\wm),\Sigma, \A(\wm, \Sigma), \mathcal{F}(\wm,\Sigma), \mathcal{B}(\wm,\Sigma), \rho'\colon \D(\wm,\Sigma) \to \<\Sigma\>^*)$, where $\rho'(D) = \rho(D)|_{\<\Sigma\>}$. 
\item $\overline{\soc}(\wm,\Sigma)$ is the localization of $\soc(\wm,\Sigma)$ at $S(\wm,\Sigma)$, as defined in Definition~\ref{def:socloc}. 
\end{enumerate}

\end{definition}

\begin{theorem} \label{thm:general}
Let $\wm$ be a normal monoid of dominant weights of $G$. Then $\wm$ is the weight monoid of a smooth affine spherical $G$-variety if and only if there exists a subset $\Sigma$ of $\Sigma^{sc}(\wm)$ such that
\begin{enumerate}
\item $\Sigma$ is adapted to $\wm$ and there is no subset of $\Sigma^{sc}(\wm)$ containing $\Sigma$ that is adapted to $\wm$; \label{general_1}
\item $\overline{\soc}(\wm,\Sigma)$ is the socle of a spherical module; and \label{general_2}
\item the $|\D(\wm,\Sigma)|$-tuple $(\rho(D))_{D \in \D(\wm,\Sigma)}$ can be completed to a basis of $\Z\wm^*$. \label{general_3}
\end{enumerate}
\end{theorem}

\begin{remark} \label{rem:thm:general}
\begin{enumerate}[(a)]
\item Recall that by Proposition~\ref{prop:adapsr}, determining the set $\Sigma^{sc}(\wm)$ is a finite, combinatorial problem. Similarly, Proposition~\ref{prop:luna-adapt}, which relies on the (proved) Luna Conjecture, reduces checking condition (\ref{general_1}) in Theorem~\ref{thm:general} to a finite, combinatorial problem.  \label{rem:thm:general:lc}
\item Verifying conditions (\ref{general_2}) and (\ref{general_3}) of Theorem~\ref{thm:general} is also a finite problem. Indeed, (\ref{general_2}) reduces to checking that $\overline{\soc}(\wm,\Sigma)$ is the sum of socles from Table~\ref{table:socles}, up to isomorphism (see Theorem~\ref{thm:soclesmodules}). By the Elementary Divisors Theorem, condition (\ref{general_3}) comes down to checking that the maximal minors of an integer matrix have greatest common divisor equal to $1$.
\item If $\Sigma^{sc}(\wm)$ is adapted to $\wm$, then $\Sigma = \Sigma^{sc}(\wm)$ is the only set that satisfies condition (\ref{general_1}) in Theorem~\ref{thm:general}. By Remark~\ref{rem:sigmascadapted} this is the case when $\wm$ is $G$-saturated or when $\Sigma^{sc}(\wm)$ does not contain any simple roots.  
\end{enumerate}
\end{remark}

\begin{proof}[Proof of Theorem~\ref{thm:general}]
We first show that the conditions on $\wm$ in the Theorem are necessary for $\wm$ to be smooth. Let $X$ be a smooth affine spherical $G$-variety $X$ such that $\wm(X)=\wm$. We put $\Sigma  = \Sigma^{sc}(X)$. Then condition (\ref{general_1}) holds by the definition of `adapted' and by Proposition~\ref{prop:smoothgeneric} and Corollary~\ref{cor:genmaxadap}. Conditions (\ref{general_2}) and (\ref{general_3}) follow from Theorem~\ref{thm:camus_smoothnes_crit} with $Y$ equal to the unique closed $G$-orbit of $X$ because 
\begin{enumerate}[(i)]
\item $S(\wm,\Sigma)$ is equal to $S'$ of Theorem~\ref{thm:camus_smoothnes_crit}; \label{item:proofgen1}
\item $\overline{\soc}(\wm,\Sigma)$ is equal to the localization of $\soc(X_{Y,G})$ at $S'$; and \label{item:proofgen2}
\item the $|\D_{X_{Y,G}}|$-tuple $(\rho_X(D))_{D\in\D_{X_{Y,G}}}$ is equal to the $|\D(\wm,\Sigma)|$-tuple $(\rho(D))_{D \in \D(\wm,\Sigma)}$. \label{item:proofgen3}
\end{enumerate}  
The three claims (\ref{item:proofgen1}), (\ref{item:proofgen2}) and (\ref{item:proofgen3}) are consequences of standard facts in the combinatorial theory of spherical varieties, which can be found in \cite{knop-lv}, \cite{luna-typeA} and \cite{timashev-embbook}, together with our analysis in Section~\ref{sec:combinatorics}. More specifically, the invariants in Definition~\ref{def:thminvariants} are the combinatorial descriptions of certain geometric invariants of $X$: 
\begin{enumerate}[1.]
\item $\s(\wm, \Sigma) = \s(X)$, by the uniqueness statement in Proposition~\ref{prop:adapt-spher-roots}: indeed $\s(\wm,\Sigma)$ satisfies properties (\ref{item:1})-(\ref{item:4}) of Proposition~\ref{prop:adapt-spher-roots} by construction (see the proof of Proposition~\ref{prop:luna-adapt}), and $\s(X)$ satisfies them by the first part of Proposition~\ref{prop:adapt-spher-roots}; 
\item $\V(\wm,\Sigma) = \V(X)$;
\item $\mathcal{C}(\wm,\Sigma)$ is equal to $\mathcal{C}(X)$ of \cite{knop-lv};
\item $\mathcal{F}(\wm,\Sigma)$ is equal to $\mathcal{F}(X)$ of \cite{knop-lv}; 
\item $\mathcal{B}(\wm,\Sigma)$ is identified with $\mathcal{B}(X)$ of \cite{knop-lv}; 
\item $\D(\wm,\Sigma)$ is identified with $\D_X$;
\item $S(\wm, \Sigma)$ is the set of simple roots that do not move any of the colors of $X$ that do not contain the closed orbit;
\item $\rho \colon \D(\wm,\Sigma) \to \Z\wm^*$ is identified with $\rho_X\colon\D_X \to \lat(X)^*$;
\item $\soc(\wm,\Sigma)$ is identified with $\soc(X)$;
\item $\overline{\soc}(\wm,\Sigma)$ is identified with the localization of $\soc(X)$ at $S(\wm, \Sigma)$.
\end{enumerate}  

We now prove the sufficiency of the conditions on $\wm$. Condition~\ref{general_1} says, by Definition~\ref{def:adapted}, that there exists an affine spherical $G$-variety $X$ with $\wm(X)=\wm$ and $\Sigma^{sc}(X) = \Sigma$. Using, once again, that $\soc(X)$ and $S'$ are described combinatorially as in Definition~\ref{def:thminvariants}, the Camus smoothness criterion (i.e.\ Theorem~\ref{thm:camus_smoothnes_crit}) and conditions (\ref{general_2}) and (\ref{general_3}) imply that $X$ is smooth.    
\end{proof}

We can prove now Proposition~\ref{prop:localizroots} and Theorem~\ref{thm:main} from Section~\ref{sec:intro}.

\begin{proof}[Proof of Proposition~\ref{prop:localizroots}]
Consider the convex cone $C$ in $\Hom_Z(\wm,\Q)$ generated by $\{\alpha^\vee|_{\Z\wm}: \alpha\in S\}$. It intersects $\V(\wm)$ at least in $0$, so there are faces of $C$ whose relative interior intersect $\V(\wm)$ (at the very least, the linear part of $C$). By convexity of $\V(\wm)$, for any two of such faces $C_1,C_2$ of $C$ there is a point in the relative interior of their convex envelope lying in $\V(\wm)$, so $C$ has a face $C_3$ containing $C_1$ and $C_2$ and such that the relative interior of $C_3$ intersects $\V(\wm)$. It follows that $C$ has a maximal such face $C_m$. The set $S_\wm$ with the required properties is the maximal subset of $S$ such that $C_m$ is the convex cone generated by $\{\alpha^\vee|_{\Z\wm}: \alpha\in S_\wm\}$.
\end{proof}

\begin{proof}[Proof of Theorem~\ref{thm:main}]  \label{proof_thm_main}
Let $\wm$ be $G$-saturated. Thanks to Corollary~\ref{cor:allsimdoub_msirr}, the set $\Sigma^{sc}(\wm)$ is itself its unique maximal subset adapted to $\wm$, and for brevity let us set from now on $\Sigma := \Sigma^{sc}(\wm)$.

Recall that $\soc(\wm,\Sigma)$ is the socle of an affine spherical variety with weight monoid $\Gamma$. By Proposition~\ref{prop:lunaXGsat} this variety has no $G$-stable prime divisors, i.e.\ $\mathcal{B}(\wm,\Sigma)=\emptyset$, and $|a(\alpha)|=1$ for all $\alpha\in S\cap \Sigma$. This second fact also implies that for any color $D$ the element $\rho(D)$ is a multiple of the restriction of some simple coroot of $G$.

We begin by making the following claim: 
\begin{equation} \label{proof:saturated:socle}
\begin{split}
\overline{\soc}(\wm,\Sigma) \text{ is the socle of a spherical module} \\ \Longleftrightarrow (S_{\wm},S^p(\wm), \Sigma^N(\wm) \cap \Z S_{\wm}) \text{ is admissible.}
\end{split}
\end{equation}
Indeed, since $\wm$ is $G$-saturated, $\soc(\wm,\Sigma)$ is uniquely determined by its first two components. Together with $S_{\wm}$ these two components of the localization $\overline{\soc}(\wm,\Sigma)$ of $\soc(\wm,\Sigma)$ build up the triple in the theorem, modulo the difference between $\Sigma^N(\wm)$ and $\Sigma^{sc}(\wm)$ which is  handled by Proposition~\ref{prop:Nspherroots_Xwm}.

Since it is true for $\soc(\wm,\Sigma)$, also for $\overline{\soc}(\wm,\Sigma)$ the component $\mathcal B$ is empty and all its colors are multiples of coroots. Then it is enough to observe that the list of primitive admissible triples is obtained from Table~\ref{table:socles} by only keeping the socles that satisfy these two conditions. This proves claim (\ref{proof:saturated:socle}).

To prove Theorem~\ref{thm:main} we can assume that one (whence also the other) of the equivalent statements  in claim (\ref{proof:saturated:socle}) holds. 
An immediate consequence of this assumption is that
\begin{equation} \label{proof:saturated:nosimpordouble}
\text{for all }\alpha\in S_\wm\setminus S^p(\wm) \text{ neither }\alpha \text{ nor }2\alpha \text{ is in }\Sigma,
\end{equation}
because this holds for all the six socles in Table~\ref{table:socles} that correspond to primitive admissible triples as above.

At this point the theorem follows from Theorem~\ref{thm:general} if we prove that
\begin{enumerate}
\item \label{proof:saturated:basis} the set $\rho(\D(\wm,\Sigma))$ is equal to $\{\alpha^{\vee}|_{\Z\wm}\colon \alpha \in S_{\wm}\setminus S^p(\wm)\}$, and
\item \label{proof:saturated:injective} $\rho$ is injective on $\D(\wm,\Sigma)$ if and only if condition (b) in Theorem~\ref{thm:main} is met. 
\end{enumerate}
We start with claim (\ref{proof:saturated:basis}). For all $\alpha\in S\cap \Sigma$ we have $|a(\alpha)|=1$, so a color of the spherical system $\s(\wm, \Sigma)$ is moved by more than one simple root if and only if it's exactly two orthogonal simple roots $\alpha,\beta$, which then add up to an element of $\Sigma$ and whose coroots are equal on $\Z\wm$. It follows that assigning to a simple coroot $\alpha^\vee$ the element $\rho(D)$ for $D$ equal to a color moved by $\alpha$ induces a well-defined bijection between $\{\alpha^{\vee}|_{\Z\wm}\colon \alpha \in S\}$ and $\rho(\Delta(\wm,\Sigma))$. This bijection is the identity, except for the cases where $\alpha$ or $2\alpha$ is in $\Sigma$.

We deduce that $S_\wm=S(\wm, \Sigma)$, and that $\{\alpha^{\vee}|_{\Z\wm}\colon \alpha \in S_{\wm}\setminus S^p(\wm)\}$ is equal to $\rho(\D(\wm,\Sigma))$ thanks to \eqref{proof:saturated:nosimpordouble}. This proves claim (\ref{proof:saturated:basis}).

Before we turn to the proof of claim (\ref{proof:saturated:injective}), we observe that, if $\alpha$ and $\beta$ are two different elements of $S_{\wm} \setminus S^p(\wm)$, then 
\begin{equation} \label{proof:corootequal:orthogonal}
\alpha^{\vee}|_{\Z\wm} = \beta^{\vee}|_{\Z\wm} \Rightarrow \alpha \perp \beta.
\end{equation}
Indeed, if $\alpha$ and $\beta$ are not orthogonal, then they belong to the same connected component of the Dynkin diagram of $S_{\wm}$, so they appear in the same primitive admissible triple. But, among the primitive admissible triples given in Definition~\ref{def:admissible}, only in case \ref{item:admiss:AnAk}.\ does the set $S_{\wm} \setminus S^p(\wm)$ contain two non-orthogonal simple roots, and for any such couple the two corresponding simple coroots take different values on some element of $\Sigma^N(\wm) \cap \Z S_{\wm}$.   

Finally we prove claim (\ref{proof:saturated:injective}). We first assume that $\rho$ is injective. Let $\alpha,\beta \in S_{\wm} \setminus S^p(\wm)$ with $\alpha^{\vee}|_{\Z\wm} = \beta^{\vee}|_{\Z\wm}$ and $\alpha\neq\beta$. Then, by (\ref{proof:saturated:basis}) and the combinatorial definition of the colors of $\s(\wm, \Sigma)$, the simple roots $\alpha$ and $\beta$ move the same color of $\s(\wm, \Sigma)$, so they are orthogonal and $\alpha + \beta \in \Sigma^N(\wm)$. In particular $\alpha+\beta\in\Z\wm$. This proves part (\ref{cond:sumisNsphericalroot}) of Theorem~\ref{thm:main}. 

For the reverse implication, we assume part (\ref{cond:sumisNsphericalroot}) of Theorem~\ref{thm:main}. Let $D,E \in \D_{(\wm,\Sigma)}$ be such that $\rho(D) = \rho(E)$, and let $\alpha$ and $\beta$ be simple roots in $S_{\wm} \setminus S^p(\wm)$ moving resp.\ $D$ and $E$. Then (\ref{proof:saturated:basis}) says that $\alpha^{\vee}|_{\Z\wm} = \beta^{\vee}|_{\Z\wm}$. It follows from (\ref{cond:sumisNsphericalroot}) of Theorem~\ref{thm:main} that $\alpha+\beta \in \Z\wm$. Since $\alpha \perp \beta$ by \eqref{proof:corootequal:orthogonal}, it follows from Proposition~\ref{prop:adapnsphroots} that $\alpha+\beta \in\Sigma^N(\wm)$, and from the definition of the colors of $\s(\wm, \Sigma)$ that $D=E$. This completes the proof. 
\end{proof}

\begin{example}\label{ex:need_tuple}
This example shows that condition (\ref{cond:sumisNsphericalroot}) in Theorem~\ref{thm:main} cannot be removed --- and therefore also that it is not possible to replace, in Theorem~\ref{thm:general}, the $|\D(\wm,\Sigma)|$-tuple $(\rho(D))_{D \in \D(\wm,\Sigma)}$ with the {\em set} $\rho(\D(\wm,\Sigma))$. Let $G=\SL(3)\times\SL(3)$ and $\wm = \N\{\omega_1+\omega_1'\}$. Then $S^p(\wm) = \{\alpha_2, \alpha_2'\}$. One checks that $\Sigma^N(\wm) = \emptyset$ and that $S_{\wm} = S$. It follows that the triple $(S_\wm, S^{p}(\wm), \Sigma^N(\wm))$ is admissible. Moreover, condition~(\ref{cond:partofbasis}) of Theorem~\ref{thm:main} is also met. Since $\Sigma^N(\wm) = \emptyset$ there is, up to isomorphism, only one affine spherical $G$-variety with weight monoid $\wm$. It is $X_0:= \overline{G\cdot x_0} \inn V(\omega_2+\omega_2')$, where $x_0$ is a highest weight vector in $V(\omega_2+\omega_2')$. As is well-known, $X_0$ is not smooth. This shows that Theorem~\ref{thm:main} would be false without
condition~(\ref{cond:sumisNsphericalroot}).  
\end{example}

\section{Smooth affine model varieties} \label{sec:model}

In this section we apply our smoothness criterion to show Theorem~\ref{thm:model}. Here we focus on the monoid $\Gamma = \Lambda^+$, which is $G$-saturated. For this reason we apply our criterion in the version of Theorem~\ref{thm:main}.

\begin{proposition}\label{prop:model}
Suppose that $G$ is semisimple and simply connected. Then:
\begin{enumerate}
\item\label{prop:model:Sp} $S^p(\Lambda^+)=\emptyset$,
\item\label{prop:model:Sigma} $\Sigma^N(\Lambda^+) = \{ \alpha+\alpha' \;|\; \alpha,\alpha'\in S,\; \alpha\neq\alpha',\;\alpha \not\perp\alpha'\}.$
\end{enumerate}
\end{proposition}
\begin{proof}
Part (\ref{prop:model:Sp}) is obvious. Part (\ref{prop:model:Sigma}) follows from \cite{luna-model}, let us give a direct proof here. We observe that the only elements of $\Sigma^{sc}(G)$ compatible with $S^p(\Lambda^+)=\emptyset$ are either sums $\alpha+\alpha'$ of two different simple roots, or simple roots, or doubles of simple roots. Simple roots are excluded from $\Sigma^N(\dw)$ by Proposition~\ref{prop:adapnsphroots}. The double of any simple root $\alpha$ is excluded by the same proposition, because if $\omega_{\alpha}$ is the fundamental dominant weight corresponding to $\alpha$ then $\omega_{\alpha} \in \dw$ (since $G$ is simply connected) and
$\langle\alpha^\vee,\omega_\alpha\rangle =1\notin2\ZZ$. 

Finally, let $\sigma = \alpha + \beta$ where $\alpha, \beta \in S$ with $\alpha \neq \beta$. If $\alpha$ is orthogonal to $\beta$, then $\sigma \not\in \Sigma^N(\dw)$ by Proposition~\ref{prop:adapnsphroots} because $\alpha^{\vee} \neq  \beta^{\vee}$. On the other hand, if $\alpha$ is not orthogonal to $\beta$, then $\sigma \in \Sigma^N(\dw)$, again by  Proposition~\ref{prop:adapnsphroots}.
\end{proof}

Thanks to part (\ref{prop:model:Sp}) of the proposition above, we discuss now the admissibility of $(S_{\Lambda^+},\emptyset, \Sigma^N(\Lambda^+) \cap \Z S_{\Lambda^+})$ in view of applying Theorem~\ref{thm:main}.

Part (\ref{prop:model:Sigma}) of the proposition implies that no spherical root in $\Sigma^N(\Lambda^+)$ is the sum of simple roots in different connected components of the Dynkin diagram of $G$.

It follows that the triple $(S_{\Lambda^+},\emptyset, \Sigma^N(\Lambda^+) \cap \Z S_{\Lambda^+})$ is admissible if and only if $(S'\cap S_{\Lambda^+}, \emptyset, \Sigma^N(\Lambda^+) \cap \Z(S'\cap S_{\Lambda^+}))$ is admissible for all $S'\subseteq S$ corresponding to a connected component of the Dynkin diagram of $G$.

\begin{lemma}
Under the assumptions of Proposition~\ref{prop:model}, let $S'\subseteq S$ correspond to a connected component of the Dynkin diagram of $G$.
\begin{enumerate}
\item If $S'$ is of type $\mathsf A_n$ with $n\geq 1$ odd, then $S'\cap S_{\Lambda^+}=\emptyset$.
\item If $S'$ is of type $\mathsf A_n$ with $n\geq 2$ even, then $S'\cap S_{\Lambda^+}=\{\alpha_i\;|\; i \text{ odd}\}$.
\item If $S'$ is of type $\mathsf B_n$ with $n\geq 3$, then $S'\cap S_{\Lambda^+}=S'$.
\item If $S'$ is of type $\mathsf C_n$ with $n\geq 2$, then $S'\cap S_{\Lambda^+}=\{\alpha_i\;|\; i \text{ odd}\}$.
\item If $S'$ is of type $\mathsf D_n$ with $n\geq 4$, then $S'\cap S_{\Lambda^+}=S'$. 
\item If $S'$ is of type $\mathsf E_6$, then $S'\cap S_{\Lambda^+}=\{\alpha_3,\alpha_4,\alpha_5\}$.
\item If $S'$ is of type $\mathsf E_7$, then $S'\cap S_{\Lambda^+}=S'$.
\item If $S'$ is of type $\mathsf E_8$, then $S'\cap S_{\Lambda^+}=S'$.
\item If $S'$ is of type $\mathsf F_4$, then $S'\cap S_{\Lambda^+}=\{\alpha_1,\alpha_2,\alpha_3\}$.
\item If $S'$ is of type $\mathsf G_2$, then $S'\cap S_{\Lambda^+}=S'$.
\end{enumerate}
\end{lemma}
\begin{proof}
It is elementary to check that the indicated sets are maximal such that there exists a linear combination of the corresponding simple coroots, with strictly positive coefficients, that is non-negative on $\Sigma^N(\Lambda^+)$. 
\end{proof}

\begin{corollary}
The triple $(S'\cap S_{\Lambda^+}, \emptyset, \Sigma^N(\Lambda^+) \cap \Z(S'\cap S_{\Lambda^+}))$ is admissible if and only if $S'$ has type $\mathsf A_n$ or $\mathsf C_n$.
\end{corollary}
\begin{proof}
For $S'$ of type $\sA_n$ with $n$ odd, the triple is $(\emptyset,\emptyset,\emptyset)$, which is admissible (cf.\ Remark~\ref{rem:emptytriple}).

Assume $S'$ has type $\sA_n$ with $n$ even or $\sC_n$ with arbitary $n$. Then the triple above is obtained as a ``union'', in the sense of Definition~\ref{def:admissible}, of triples of the form $(\sA_1, \emptyset, \emptyset)$, which is the second admissible triple in List~\ref{list:pat} for $n=1$.

In all other cases $S'\cap S_{\Lambda^+}$ is connected, so expressing the above triple as a union only results in a trivial decomposition. The triple itself does not appear in List~\ref{list:pat}, so these cases are excluded and the proof is complete.
\end{proof}

Since again $S^p(\Lambda^+)=\emptyset$, it remains to show that $\{\alpha^{\vee}|_{\Z\wm}\colon \alpha \in S_{\Lambda^+}\}$ is a subset of a basis of $\Z(\Lambda^+)^*$ if $G$ has simple factors of type $\sA_n$ or $\sC_n$. This is obvious, since $\{\alpha^\vee\;|\; \alpha\in S\}$ is the dual basis of the basis of $\Lambda^+$ consisting of the fundamental dominant weights of $G$.

\def\cprime{$'$} \def\cprime{$'$} \def\cprime{$'$} \def\cprime{$'$}
  \def\cprime{$'$}
\providecommand{\bysame}{\leavevmode\hbox to3em{\hrulefill}\thinspace}
\providecommand{\MR}{\relax\ifhmode\unskip\space\fi MR }
\providecommand{\MRhref}[2]{%
  \href{http://www.ams.org/mathscinet-getitem?mr=#1}{#2}
}
\providecommand{\href}[2]{#2}

\end{document}